\documentclass[english, 12pt]{article}
\usepackage{amsmath}
\usepackage{amsthm}
\usepackage{amssymb}
\usepackage{amsbsy}
\usepackage{amsfonts}
\usepackage{amstext}

\numberwithin{equation}{section}
\theoremstyle{plain}
\newtheorem{thm}{Theorem}[section]
\newtheorem{prop}[thm]{Proposition}
\newtheorem{cor}[thm]{Corollary}
\newtheorem{lem}[thm]{Lemma}
\theoremstyle{definition}
\newtheorem{exa}[thm]{Example}

\newtheorem{rem}[thm]{Remark}
\newtheorem{defi}[thm]{Definition}

\usepackage[dvips]{graphicx}
\numberwithin{equation}{section}
\DeclareMathOperator*{\real}{\mathbb{R}}
\DeclareMathOperator*{\com+}{\mathbb{C}_+}
\DeclareMathOperator*{\comp}{\mathbb{C}}
\DeclareMathOperator*{\nat}{\mathbb{N}}

\DeclareMathOperator*{\im}{\text{Im}}
\DeclareMathOperator*{\re}{\text{Re}}
\DeclareMathOperator*{\supp}{\text{supp}}
\DeclareMathOperator*{\sign}{\text{sign}}

\DeclareMathOperator*{\tor}{\mathbb{T}}

\begin{document}
\title{On monotone convolution and monotone infinite divisibility}
\author{Takahiro Hasebe \\ Graduate School of Science, Kyoto University,\\  Kyoto 606-8502, Japan \\ E-mail: hsb@kurims.kyoto-u.ac.jp}
\date{}

\maketitle 
\tableofcontents
\section{Introduction}
\subsection{Quantum probability theory}
Quantum probability theory (also called non-commutative probability theory) 
is a theory which can be seen as an algebraic foundation of quantum 
mechanics. Recently this area is developing more and more from the 
mathematical point of view. 

The basis of quantum probability theory is a pair $(\mathcal{A}, \phi)$ called a noncommutative probability space, where 
$\mathcal{A}$ is a unital $*$-algebra on $\comp$ and $\phi$ is a state on $\mathcal{A}$. 
One of the most important inner structure of the pair $(\mathcal{A}, \phi)$ 
is the notion of independence. In many cases, independence determines one theory. If we consider free independence, the corresponding theory is called free probability theory. 
 
 There have been two big developments in this area: free probability theory started by Voiculescu 
and the analyses on a Boson Fock space developed by Hudson and Parthasarathy.  
These two areas can be characterized by independence: 
free independence for free probability theory, and Bose (Boson) independence for Hudson-Parthasarathy theory. 

Hudson and Parthasarathy have discussed mainly properties of creation and annihilation operators 
on the Boson Fock space over $L^2 (\real)$ \cite{H-P1}. In this area, there are many interesting developments  
which could not be explained briefly. We only note here that a Boson-Fermion 
correspondence has been established in terms of stochastic integration \cite{H-P2} in case of real one dimension. 

Free probability theory has been founded for the study of the type I\hspace{-.1em}I$_1$ factor which is 
a group algebra of a free group \cite{Voi}. In this area, there is a surprising connection to random matrix theory \cite{V1}. 
This result has given a strong impetus to the later developments. For instance, 
free entropy is defined by the limit of the Boltzmann entropy of random matrix theory. 
Large deviations of random matrices are important in free entropy theory. 
An additional example is stochastic integration theory. 
One can understand a stochastic integration in free probability as the limit of a stochastic integration 
in random matrix theory \cite{Bia1}.

In addition to the above two notions of independence, there are other notions of independence such as Fermi independence, monotone independence, anti-monotone independence and boolean independence. 
Muraki has clarified in \cite{Mur4, Mur5} (see also \cite{Fra2}) that there are only five ``nice'' notions of independence, i.e.,  
boson, monotone, anti-monotone, free, boolean. 

In addition to independence, 
the Fock space structure is also important in quantum probability theory. 
The free Fock space (or full Fock space) over $L^2(\real)$ is defined by 
\[
\Gamma_f (L^2 (\real)) := \comp \oplus \bigoplus_{n = 1} ^{\infty} L^2 ( \mathbb{R} ^n). 
\] 
The Boson Fock space is defined as a symmetrized Fock space: 
\[
\Gamma_b (L^2 (\real)):= \comp \oplus \bigoplus_{n = 1} ^{\infty} L^2 ( \mathbb{R} ^n)_s,
\] 
where $L^2 (\real^n)_s$ means the set $\{f \in L^2(\real^n); f \text{ is symmetric } \}$. 
The monotone Fock space is defined by
\[
\Gamma_m(L^2(\real)):= \comp \oplus \bigoplus_{n = 1} ^{\infty} L^2 ( \mathbb{R} ^n _{>}), 
\]
where $\real^n _{>} := \{(x_1, \cdots, x_n) \in \real^n; x_1 > x_2 > \cdots > x_n \}$. 
A Fock space is important in understanding a Brownian motion. 
In each Fock space $\Gamma _X (L^2(\real))$, a Brownian motion is defined by the operator 
\[
B_X (t):=a_X(1_{[0,t)})+ a_X ^* (1_{[0,t)}), 
\]   
where $a_X (f)$ and $a^* _X(f)$ are defined by  \\
(1) in the full Fock space or the monotone Fock space, 
\[
a_X ^* (f)f_1 \otimes \cdots \otimes f_n := f \otimes f_1 \otimes \cdots \otimes f_n, ~~ X = f, m; 
\]
(2) in the Boson Fock space, 
\[
a_b ^* (f)g^{\otimes n} := \sqrt{n} f \hat{\otimes}g^{\otimes n}. 
\] 
$a_X (f)$ is defined by the adjoint operator of $a_X ^* (f)$ w.r.t. the inner product of each Fock space. 
$a_X (f)$ and $a_X ^* (f)$ are bounded operators for $X = m$ or $f$ and 
unbounded for $X = b$.

The unital $*$-algebra $\mathcal{A}$ constituting the non-commutative probability space $(\mathcal{A}, \phi)$ is taken to be the $*$-algebra generated by $a_X (f)$, $f \in L^2(\real)$
on each Fock space; the state is taken to be the vacuum state. 
It is important that the Brownian motions defined above have independent increments. We explain this point in the case of the monotone Fock space. 

\begin{defi}
Let $\mathcal{A}$ be a $*$-algebra and let $\phi$ be a state. \\
(1) Let $\{\mathcal{A}_m \}_{m = 1} ^{n}$ be a sequence of 
$*$-subalgebras in $\mathcal{A}$. Then $\{\mathcal{A}_m \}_{m = 1} ^{n}$ is said to be \textit{monotone independent} if the following condition holds.  
\begin{equation}
\begin{split}
\phi(a_1  a_2  \cdots a_n ) &= \phi (a_k ) \phi (a_1  a_2  \cdots \check{a_k} \cdots  a_n ) \\
& \text{ if } a_m \in  \mathcal{A}_{i_m} \text{ for all } 1 \leq m \leq n \text{ and $k$ satisfies } i_{k-1}< i_k > i_{k+1}.
\end{split}
\end{equation}
If $k = 1$ (resp. $k = n$), the above inequality is understood to be $i_1 > i_2$ (resp.  $i_{n-1} < i_{n}$). \\
$(2)$ Let $\{b_i \}_{i = 1} ^{n}$ be a sequence of elements in $\mathcal{A}$. $\{b_i \}_{i = 1} ^{n}$ is said to be monotone independent if 
the $*$-algebras $\mathcal{A}_{i}$ generated by each $b_i $ without unit form a monotone independent family. 
\end{defi} 

\begin{thm}\cite{Mur1}\label{indep}
The Brownian motion on the monotone Fock space has independent increments w.r.t. the vacuum: for any $0 < t_1 < \cdots < t_n < \infty$, 
\[ 
B_m(t_2)- B_m (t_1), \cdots, B_m(t_n) -B_m (t_{n-1})
\]
are monotone independent.
\end{thm}

The above constructions of the three types of Brownian motions can be seen from a more general point of view. 
In one dimension, a generalized Fock space is defined by replacing the Hermite polynomials with arbitrary orthogonal polynomials; 
in the case of infinite dimensions, some construction is known \cite{A-C-L, Lu}. Such a Fock space is called an interacting Fock space. 
In the case of infinite dimensions, it seems that there is a connection between 
interacting Fock spaces and orthogonal polynomials only in some special cases. We mention three important examples. 
In the usual Boson Fock space over $L^2(\real)$, 
Hermite polynomials are well known to be connected deeply to the Brownian motion through the Wiener-It\^{o} isomorphism. 
Also in the case of free probability theory, an analogue of Wiener-It\^{o} isomorphism is considered in \cite{Bia1}. 
Chebysheff polynomials of the second kind are effectively used in Malliavin Calculus of free probability theory \cite{B-S}. 
The last example is a $q$-deformed Fock space called a $q$-Fock space. 
$q$-Hermite polynomials appear naturally on a $q$-Fock space \cite{B-K-S}. When $q = 0$, a $q$-Fock space becomes a free Fock space; this fact implies that the corresponding non-commutative $q$-probability theory can be seen as a generalization of free probability theory. When we consider a $q$-Fock space, the corresponding Brownian motion 
is called a $q$-Brownian motion and is analyzed in \cite{B-K-S}. 
An interacting Fock space sometimes loses a connection with 
independence, but is still interesting in some aspects such as central limit theorem \cite{A-C-L}.       
 
Monotone probability theory should be connected to Chebysheff polynomials 
of the first kind, but such a connection has not been clarified so much yet. This is an interesting 
direction of research on monotone probability theory.

In this paper, we develop an analysis of the monotone independence (especially monotone convolution). 
This work will be important when we try to clarify special features of monotone independence contrasted with various other notions of independence. At the same time, we aim to clarify common properties among various notions of independence. This work is also expected to have connections with an operator theoretic approach \cite{Fra3, F-M} 
or a categorical approach \cite{Fra1}.  

\subsection{Monotone probability theory and main results of this paper}
Muraki has defined the notion of monotone independence in \cite{Mur3} as an algebraic structure of the monotone Fock space \cite{Lu, Mur1}, and then defined monotone convolution as the probability distribution of the sum of two  
monotone independent random variables. 
Analysis of monotone convolution has been developed by Muraki \cite{Mur3}, 
where the viewpoint of harmonic analysis is emphasized. 
 
The reciprocal Cauchy transform is defined by
\begin{equation}
 H_{\mu} (z): =\frac{1}{G_\mu (z)},
\end{equation}
where $G_\mu$ is the Cauchy transform (or Stieltjes transform)
\begin{equation}
G_\mu (z) = \int_{\real} \frac{1}{z-x}d\mu(x).
\end{equation} 
$H_\mu$ is analytic and maps the upper half plane into itself. Moreover, $\inf_{\im z > 0}\frac{\im H_\mu(z)}{\im z} = 1$. 
Consequently, $H_\mu(z)$ can be expressed uniquely in the form
\begin{equation}\label{recip1}
H_\mu(z) = z + b + \int_{\real} \frac{1 + xz}{x-z}\eta(dx), 
\end{equation}
where $b \in \real$ and $\eta$ is a positive finite measure. The reader is referred to \cite{Akh}.    

The monotone convolution $\mu \rhd \nu$ of two probability measures $\mu$ and $\nu$ is 
characterized by the relation 
\begin{equation}\label{mur11}
H_{\mu \rhd \nu}(z) = H_{\mu}(H_{\nu}(z)). 
\end{equation}
This relation naturally allows us to extend monotone convolution to probability measures with unbounded supports. Recently, Franz \cite{Fra3} has clarified the notion of monotone independence of unbounded operators.  

Similarly to the classical convolution, one can define the notion of 
infinitely divisible distributions. Such a distribution is called a
$\rhd$-infinitely divisible distribution. 
When we consider only probability measures with compact supports, 
there is a natural one-to-one correspondence among  
a $\rhd$-infinitely divisible probability measure, 
a weakly continuous one-parameter monotone convolution semigroup 
of probability measures, 
and a vector field on the upper half plane \cite{Mur3}.
The complete correspondence has been proved by Belinschi \cite{Bel}. 

\begin{thm}\label{inf.div}
There is a one-to-one correspondence among the following four objects: 
\begin{itemize}
\item[$(1)$] a monotone infinitely divisible distribution $\mu$; 
\item[$(2)$] a weakly continuous monotone convolution semigroup $\{\mu_t \}$ with  $\mu_0 = \delta_0, \mu_1 = \mu$; 
\item[$(3)$] a composition semigroup of reciprocal Cauchy transforms $\{H_t \}$ $(H_t \circ H_s = H_{t+s})$ with $H_0 = $\text{id}, $H_1 = H_\mu$, 
 where $H_t(z)$ is a continuous function of $t \geq 0$ for any $z \in \comp \backslash \real$; 
\item[$(4)$] a vector field on the upper halfplane which has the form $A(z) = -\gamma + \int_{\real} \frac{1 + xz}{x - z} d\tau(x)$, where $\gamma \in \real$ and $\tau$ is a positive finite measure. (This is the L\'{e}vy-Khintchine formula in monotone probability theory.) 
\end{itemize}
\end{thm}
The correspondence of (3) and (4) is obtained through the following ordinary differential equation (ODE): 
\begin{equation}\label{ODE}
\begin{split}
&\frac{d}{dt}H_t(z) = A(H_t(z)), \\
&H_0(z) = z,
\end{split} 
\end{equation}
for $z \in \comp \backslash \real$. The fact that the solution does not 
explode in finite time has been proved in \cite{Be-Po}. 

In this paper, we focus on properties of monotone convolution and monotone convolution semigroups. 
The contents of each section are as follows. 

In Section \ref{inj}, we study the injectivity of the reciprocal Cauchy transforms of $\rhd$-finitely divisible and $\rhd$-infinitely divisible distributions. 
In Section \ref{atom}, we show an interlacing property of the monotone convolution of atomic measures (Theorem \ref{atom1}) and then we conclude that the monotone convolution of atomic measures 
with $m$ and $n$ atoms contains just $mn$ atoms (Corollary \ref{atom2}). In addition, motivated by the 
study in Section \ref{inj}, we clarify that the existence of an atom in a $\rhd$-infinitely divisible distribution puts a restriction on the distribution (Theorem \ref{atom3}). 
 In Section \ref{basic}, we prove a condition for a probability measure to be supported on the positive real line, 
and show how moments change under the monotone convolution.  
In Section \ref{Diff}, we derive a differential equation about the minimum of support of a monotone convolution semigroup. 
In Section \ref{time}, we study how a property of a monotone convolution semigroup changes with respect to time parameter. 
Time-independent property is a property of a convolution semigroup which is determined at an
 instant. We show that the following properties are time-independent: the symmetry around $0$; the concentration of a support on the positive real line; 
the lower boundedness of a support; the finiteness of a moment of even order. All these properties are also time-independent in classical convolution semigroups.  
In Section \ref{stable} we classify strictly $\rhd$-stable distributions 
(or equivalently, $\rhd$-infinitely divisible and self-similar distributions). 
The result is very similar to the free and boolean cases. 
In Section \ref{connec}, a monotone analogue of the Bercovici-Pata bijection is defined. Many time-independent properties in the previous section can be formulated in terms of the Bercovici-Pata bijection.  
In Section \ref{connect1}, we clarify that the Aleksandrov-Clark measures can be represented as monotone convolutions. As a result, 
we can apply spectral analysis of a one-rank perturbation of a self-adjoint operator to monotone convolutions.  
In Section \ref{time2}, we study convolution semigroups in free probability and Boolean probability. A remarkable point is that the concentration of the support on the positive real line is a time-independent property in the monotone, Boolean and classical cases, but 
this is not true in free probability.

\section{Injectivity of reciprocal Cauchy transform}\label{inj}
For a $\rhd$-infinitely divisible distribution $\mu$, the injectivity of 
$H_{\mu}$ follows from the uniqueness of the solution of the ordinary 
differential equation (\ref{ODE}).  
This injectivity can be seen as the counterpart of the classical fact 
that for any infinitely divisible distribution, its Fourier transform has 
no zero point on $\real$. 
The result in \cite{Bel} implies that $H_{\mu}$ is injective for any $\rhd$-infinitely 
divisible distribution $\mu$ (the support of which may be unbounded).  
If a probability distribution is of finite variance, however,  
the injectivity property can be shown in a way different from \cite{Bel}. 
We do not need to embed a probability measure in a convolution semigroup. 
Moreover, the method is applicable to finitely divisible distributions. 
In this section we present the proof. 

We denote by $H^n:=H \circ H \circ \cdots \circ H$ the $n$ fold composition of a map $H$ 
throughout this paper. \\

Define a set of probability measures $\Phi := \{\mu ; H_{\mu} \text{~is injective} \}$. 
We shall prove (b) and (c) of the following properties of the set $\Phi $:
\begin{itemize}
  \item[(a)] $\mu, \nu \in \Phi  \Longrightarrow \mu \rhd \nu \in \Phi$;
  \item[(b)] $\Phi $ is closed under the weak topology of probability measures;
  \item[(c)]  If $\mu$ is a $\rhd$-infinitely  divisible distribution with finite variance, then $\mu \in \Phi$;
  \item[(c')] If $\mu$ is a $\rhd$-infinitely  divisible distribution, then $\mu \in \Phi$.
\end{itemize}
The proof of (a) is simple. 
The assumption ``finite variance'' in $(c)$ is not needed if we use the result in \cite{Bel}, and hence,  
(c') holds.
These results are contained in Theorem \ref{inj1} and Proposition \ref{inj4}. 
The set $\Psi := \{\mu ;\mu \text{ is } \rhd \text{-infinitely divisible} \}$ is difficult 
to analyze except for probability measures with compact supports. For instance, 
properties (a) and (b) seem to be difficult to prove for $\Psi$. 
We have defined $\Phi $ for this reason and aim to analyze $\Phi$ instead of $\Psi $.
(c) (or (c')) is useful as a criterion for $\rhd$-infinite divisibility. 
An application of property (c') is in Theorem \ref{atom3}. 

In the classical case, it is known that 
\begin{equation}\label{infzero}
\{\mu; \mu \text{ is infinitely divisible} \} \subsetneqq \{\mu; \hat{\mu}(\xi ) \neq 0 \text{ for all } \xi \in \real \}, 
\end{equation}
where $\hat{\mu}(\xi):= \int e^{ix\xi}d\mu(x)$,~$\xi \in \real$.  
In order to construct an example of $\mu$ whose Fourier transform has no zero points but 
is not infinitely divisible, we need to make a function $f(\xi)$ such that 
$\exp(f(\xi))$ is positive definite and $\exp(\frac{1}{n} f(\xi))$ is not positive definite for 
some $n \in \nat$. Such an example is $\hat{\mu}(\xi ) = \frac{1}{2}(e^{-\frac{\xi^2}{2}} + e^{-|\xi|})$. 
This is a positive definite function and there exists a distribution $\mu$ by Bochner's theorem. 
The fact that the distribution is not infinitely divisible is shown by Corollary 9.9 in Chapter 4 
of the book \cite{St-Ha}.

In an analogy with (\ref{infzero}), the conjecture 
\begin{equation} 
\{\mu; \rhd \text{-infinitely divisible} \} \subsetneqq \{\mu; H_{\mu} \text{ is injective} \} 
\end{equation}  
comes up in the monotone case. The author has not been able so far to prove this fact.

We prepare for the proof of (b) and (c). The next proposition is taken from \cite{Maa} in a slightly more general version. 

\begin{prop}\label{maa} \text{ \normalfont\cite{Maa}} 
A probability measure  $\mu $ has a finite variance $\sigma ^2(\mu)$ if and only if $H_\mu $ has the 
representation
\begin{equation} \label{nevan}
H_\mu (z) = a + z + \int_ {\real} \frac{1}{x-z}d\rho (x), 
\end{equation}
where $a \in \real$ and $\rho$ is a positive finite measure.
Furthermore, we have $\rho (\real) = \sigma ^2(\mu)$ and $a = -m(\mu)$, 
where $m(\mu)$ denotes the mean of $\mu$ and $\sigma ^2 (\mu)$ denotes the variance of $\mu$. 
\end{prop}

\begin{defi}
$(1)$ A probability measure $\mu$ is said to be $\rhd$-$k$-divisible if 
there exists a probability measure $\mu_k$ such that $\mu = \mu_k ^{\rhd k}$. \\
$(2)$ A probability measure $\mu$ is said to be $\rhd$-infinitely divisible if for any integer  
$1 \leq k < \infty$, there exists a probability measure $\mu_k$ such that $\mu = \mu_k ^{\rhd k}$. 
We call $\mu_k $ a $k$-th root of $\mu$. \\
\end{defi}

Let $\mu$ and $\nu$ be probability measures. For each $x \in \real$ 
let $\nu_x$ (also denoted by $\nu^x$) be a probability measure defined by the equation \cite{Mur3}
\begin{equation} \label{eq21}
H_{\nu_{x}}(z)=H_{\nu}(z) - x, 
\end{equation}
and we have the representation of a monotone convolution in the form 
$\mu \rhd \nu (A) = \int_{\real} \nu_{x}(A)d\mu(x)$. 
It follows from this representation that monotone convolution is affine in the left component:
\begin{equation}
(\theta _1 \mu + \theta _2 \nu) \rhd \lambda  = \theta _1 (\mu \rhd\lambda) + 
\theta _2 (\nu \rhd \lambda) 
\end{equation}
for all probability measures $\mu, \nu$ and $\lambda $ and $\theta_1, \theta_2 \geq 0$, 
$\theta_1 + \theta_2 = 1$. It should be noted that $\mu_x$ is weakly continuous with respect to $x$.
The reader is referred to Theorem 2.5 in \cite{Maa} for the proof.
The measurability of $\mu_x(A)$ for any Borel set $A$ 
(denoted as $A \in\mathcal{B}$($\real$))  follows from the weak continuity.
In fact, for an open set $A$, the function $x \longmapsto \mu_x(A)$ is 
lower semicontinuous, and hence, is measurable. Define the set 
$\mathcal{F}$:= \{$A \in \mathcal{B}$($\real$);  $x \longmapsto \mu_x(A)$ is measurable \}.
Every open set is contained in $\mathcal{F}$ and $\mathcal{F}$ is a $\sigma $-algebra; therefore,  $\mathcal{F}= \mathcal{B}(\real)$.

The next lemma is almost the same as Lemma 6.3 in \cite{Mur3}.

\begin{lem}\label{inj3}
Assume that a probability measure $\mu$ has finite variance and that $\mu$ is $\rhd$-$k$-divisible. Then a $k$-th root $\mu_k$ of $\mu$ has finite variance. Therefore, $\mu_k$ has the integral representation in the form 
\begin{equation} 
H_{\mu_k}(z) = a _k + z + \int_{\real} \frac{1}{x-z}d\rho_k(x).
\end{equation}
Moreover, it holds that $a_k = \frac{1}{k}a$ and $\rho_k(\real) = \dfrac{\rho(\real)}{k}$,
where $(a,\rho)$ is a pair which appears in the representation (\ref{nevan}).
\end{lem}
\begin{proof}
The monotone convolution $\mu = \mu_k ^{\rhd k}$ 
can be expressed as 
\[ \mu (A) = \int_{\real}\mu_{k,x}(A) d\mu_k ^{\rhd k-1} (x). \]
Since $\mu $ has finite variance, we have

\begin{equation}
\begin{split}
\int_{\real} y^2 d\mu(y) = \int_{\real}d\mu_k ^{\rhd k-1}(x)\int_{\real}y^2 d\mu_{k,x}(y) < \infty. 
\end{split}
\end{equation}
Hence there exists some $ y_0 \in \real$ such that $\sigma ^2 (\mu_{k,y_0}) < \infty$.
By Proposition \ref{maa}, we obtain the representation 
\[ H_{\mu_{k,y_0}} (z) = b_k + z + \int_ {\real} \frac{1}{x-z}d\rho_k (x), \]
and the representation for   $H_{\mu_k}$ 
\[ H_{\mu_{k}} (z) = b_k + y_0 + z + \int_ {\real} \frac{1}{x-z}d\rho_k (x). \]
Therefore, we have $\sigma^2(\mu_k) < \infty$ again by Proposition \ref{maa}.

Next we have 
\begin{equation}
\begin{split}
H_{\mu}(z) &= H_{\mu_k}(H^{k-1}_{\mu_k}(z))  \\
           &= a_k + H_{\mu_k} ^{k-1}(z) + \int_{\real} \frac{1}{x - H_{\mu_k ^{\rhd k-1}}(z) }\rho_k (dx)   \\
           &= a_k +  H_{\mu_k} ^{k-1}(z) - \int_{\real} G_{(\mu_k ^{\rhd k-1})_x}\rho_k (dx)       \\
           &= a_k +  H_{\mu_k} ^{k-1}(z) - \int_{\real}\rho_k (dx) \int_{\real}\frac{(\mu_k ^{\rhd k-1})_x(dy)}{z-y}      \\
           &= a_k +  H_{\mu_k} ^{k-1}(z) + \int_{\real} \frac{1}{y-z}\int_{\real}(\mu_{k}^{\rhd k-1})_x(dy)\rho_k (dx)   \\
           &= a_k + a_k + H_{\mu_k} ^{k-2}(z) + \int_{\real}\frac{1}{y-z}  \Bigg{(}\int_{\real}(\mu_{k}^{\rhd k-1})_x(dy)\rho_k (dx) + \int_{\real}(\mu_{k}^{\rhd k-2})_x(dy) \rho_k (dx) \Bigg{)}    \\
           &= \cdots \\ 
           &= ka_k + z + \int_{\real} \frac{1}{y-z}\Bigg{(} \sum_{m = 0} ^{k-1}\int_{\real}(\mu_k ^{\rhd m})_x (dy) \rho_k(dx) \Bigg{)}, 
\end{split}
\end{equation}
where $\mu_k ^{\rhd 0} := \delta_0$. 
From the uniqueness of the representation, we obtain $a = ka_k$ and 
\begin{equation}
\rho(dy) = \sum_{m = 0} ^{k-1} \int_{\real}(\mu_k ^{\rhd m})_x(dy) \rho_k(dx),   
\end{equation} 
Hence we have $\rho(\real)=k\rho_k(\real)$.  
\end{proof}
\begin{thm}\label{inj1}
Let $\mu$ be a probability measure with finite variance. \\
$(1)$ Assume that $\mu$ is $\rhd$-$n$-divisible. If $z_1 \neq z_2$ are two points in $\com+$ satisfying 
$\im z_1 \cdot \im z_2 > \dfrac{\rho(\real)}{n}$, then $H_\mu(z_1) \neq H_\mu(z_2)$.
In particular, $H_\mu$ is injective in $\{z \in \com+; \im z > \sqrt{\dfrac{\rho(\real)}{n}}\}$.
Moreover, the constant $\dfrac{\rho(\real)}{n}$ is optimal.  \\
$(2)$ Assume that $\mu$ is $\rhd$-infinitely divisible. Then $H_{\mu}$(and hence $G_\mu$) is injective. \\
\end{thm}  
\begin{proof} 
(1) We use the same notation for the integral representation of $\mu$ and $\mu_k$ as the one adopted in the previous 
lemma. 
Pick an arbitrary real number $r  < 1$ and fix it. Let $z_1, z_2$ be any two points satisfying 
$\frac{\rho(\real)}{n\text{Im}z_1\text{Im}z_2} <r$.

First we have   
\begin{equation}
\begin{split}
| H_{\mu_n}(z_1)-H_{\mu_n}(z_2)| &= \Bigg{|} z_1 - z_2 + \int_{\real}\Bigg{(}\frac{1}{x-z_1}-\frac{1}{x-z_2}\Bigg{)}d\rho_n(x) \Bigg{|}             \\  
                                 &\geq  |z_1 - z_2| - \Bigg{|}\int_{\real}\frac{z_2 - z_1}{(x-z_1)(x-z_2)}d\rho_n(x) \Bigg{|}             \\
                                 &\geq  |z_1 - z_2| - \Bigg{|}\int_{\real}\frac{|z_2 - z_1|}{\text{Im}z_1\text{Im}z_2}d\rho_n(x) \Bigg{|}          \\
                                 &\geq  |z_1 - z_2|(1-r). 
\end{split}
\end{equation}

Since $\text{Im}H_{\mu_n}(z)\geq \text{Im}z$ for all $z \in \com+$, we can iterate the inequality:

\begin{equation}
\begin{split}
|H_\mu(z_1)-H_\mu(z_2)| &= |H^{n}_{\mu_n}(z_1)-H^{n}_{\mu_n}(z_2)|         \\
                        &\geq  |z_1 - z_2|(1-r)^n. 
\end{split}
\end{equation}
Therefore, $z_1 \neq z_2$ implies $H_\mu(z_1) \neq H_\mu(z_2)$ since $r$ can be taken arbitrary near to 1.

The optimality of the constant $\dfrac{\rho(\real)}{n}$ will be proved in Example \ref{exa2} shown later. 

\noindent
(2) For any $z_1, z_2 \in \com+$ we take $n$ large enough so that
   $\text{Im} z_1 \cdot \text{Im} z_2 > \dfrac{\rho(\real)}{n}$, then we can use the result (1). 
\end{proof}
\begin{exa} 
$H_\mu$ (or $G_\mu$) of the following probability measures are all injective: 
\begin{itemize}
\item[$(1)$] Arcsine law 
$d\mu(x) = \frac{1}{\pi \sqrt{2-x^2}}~ 1_{(-\sqrt{2},\sqrt{2})}(x)dx$, 
$H_{\mu}(z) = \sqrt{z^2 - 2}$,   
\item[$(2)$] Uniform distribution $d\mu = \frac{1}{b-a}1_{(a,b)}(x)dx$, $G_{\mu} (z) = \frac{1}{b-a}\log\Big{(} \frac{z-a}{z-b} \Big{)}$, 
\item[$(3)$] Wigner's semicircle law $d\mu(x) = \frac{1}{2\pi}\sqrt{4-x^2}1_{(-2,2)}(x)dx$, 
$H_\mu(z) = \frac{z + \sqrt{z^2- 4}}{2}$,  
\item[$(4)$] Normal distribution $d\mu(x) = \frac{1}{\sqrt{2\pi}} e^{-\frac{x^2}{2}}dx$.
\end{itemize} 
The injectivity in the cases (1), (2) and (3) can be confirmed directly. To prove the injectivity of the Stieltjes transform of the normal distribution, 
we use a general criterion for injectivity proved by Aksent'ev, which is also applicable to (1), (2) and (3).  
The reader is referred to a survey article \cite{Av-Ak} for details. 
\begin{thm}\label{univalent} (Aksent'ev) 
Let $a < c < b$ and let $p:[a, b] \to [0, \infty)$ be a function which is not identically zero, does not decrease in the interval $(a, c)$ 
and does not increase in the interval $(c, b)$. Then the function $\int_a ^b \frac{1}{z-x} p(x)dx$ is injective in $\comp \backslash [a, b]$.  
\end{thm} 
When we apply this theorem to the normal distribution $\mu$, first we restrict the distribution to the closed interval $[-n, n]$, 
which we denote by $\mu_n$,   
and then take the limit $n \to \infty$. By Theorem \ref{univalent}, $G_{\mu_n}$ is injective in $\com+$. Since $\mu_n \to \mu$ weakly, $G_\mu$ is injective 
in $\com+$ by Proposition \ref{inj4} shown later. 

Arcsine law is the only distribution known to be $\rhd$-infinitely divisible 
in the above examples. It is an interesting question whether the other examples 
are $\rhd$-infinitely divisible or not. 
\end{exa}

\begin{exa} \label{exa2}
Next we treat atomic measures. 
We define $\nu := \lambda_1 \delta_a + \lambda_2 \delta_b$ with $\lambda_1 + \lambda_2 = 1$  and $a \neq b$.
Its Cauchy transform is
\begin{equation}
\begin{split}
G_{\nu}(z) &= \frac{\lambda_1}{z-a} + \frac{\lambda_2}{z-b}   \\
           &= \frac{z-(\lambda_2 a + \lambda_1 b)}{(z-a)(z-b)}.
\end{split}
\end{equation}
For simplicity, we consider the case $b = -a$, $a > 0$ and $\lambda_1 = \lambda_2 = \frac{1}{2}$.
Then $H_\nu(z) = \frac{z^2-a^2}{z}$,  $\sigma^2(\nu) = a^2$, $m(\nu) = 0$. 
By Proposition \ref{maa}, $\rho(\real) = a^2$.  
Take $z_1= si$ and $z_2 = ti$ such that $ st= a^2$. For instance,  take 
$z_1= \frac{a}{2}i$ and $z_2 = 2ai$. Clearly we have $z_1 \neq z_2$ and 
$\im z_1 \im z_2 = a^2$. Moreover, one can see that $H_\nu(z_1) = H_\nu (z_2)$.  
Therefore, $\nu$ is not 2-divisible by Theorem \ref{inj1}. Moreover, the optimality of 
the constant $\dfrac{\rho(\real)}{n}$ is proved by the example $\nu^{\rhd n}$. 
In fact, for any integer $n$, it holds that $\sigma^2(\nu^{\rhd n}) = na^2$ 
and $m(\nu^{\rhd n}) = 0$ by Lemma \ref{inj3}.
If we take $z_1= \frac{a}{2}i$ and $z_2 = 2ai$ again, then 
$H_{\nu^{\rhd n}} = H^n_{\nu}$ maps $z_1$ and $z_2$ to the same point.
Hence the proof of Theorem \ref{inj1} has been completed.  

It is clear that $\nu^{\rhd 2}$ is 2-divisible. In addition, it is not difficult 
to prove that $\nu^{\rhd 2}$ is not 3-divisible in application of Theorem \ref{inj1}.
\end{exa}

We have seen the divisibility of atomic measures through an example. 
There is a question whether $H_{\nu}$ for $\nu = \sum_{k=1}^m \lambda _k \delta_{a_k}$ 
is $\rhd$-infinitely divisible or not. The answer is given in Section \ref{atom}, Theorem \ref{atom3}.      

In the classical probability theory, the set of infinitely divisible distributions is 
closed under the weak topology \cite{Sat}. In monotone probability theory, however, 
this is difficult to prove and the proof is unknown. 
Instead we show that the injectivity property is conserved under the weak topology. 
The proof of the next Lemma is the analogy of the case of characteristic functions, 
but the tightness of probability measures is not needed. 
Hence we can give a proof without Prohorov's theorem. 

\begin{lem}\label{inj2}
If a sequence of positive finite measures $\{\mu_n \}$ converges weakly to 
a positive finite measure $\mu$, then the Cauchy transform $G_{\mu_n}$ 
converges to $G_{\mu}$ locally uniformly on $\com+$. 
\end{lem}
\begin{proof}
Pointwise convergence follows from the definition of the weak convergence of $\{\nu_n \}$.
Locally uniform convergence is a consequence of Montel's theorem. 
\end{proof}
\begin{prop}\label{inj4} 
Let  $\{\mu_n \}$ be a sequence of positive finite measures whose $G_{\mu_n}$ are injective. If $\mu_n$ converges weakly to a nonzero positive finite measure $\mu$, then $G_\mu$ is injective. 
\end{prop}
\begin{proof}
This fact comes from Lemma \ref{inj2} and the fact that the set of injective analytic functions on a domain 
is closed under the locally uniform topology (see Section 6 of Chapter 9 in \cite{Nev}). Then the limit function is also injective 
on the domain. 
\end{proof}

After we stated some properties about the injectivity of $H_{\mu}$, it is natural 
to ask when $H_{\mu}$ becomes a diffeomorphism.
We prove the simple characterization of $\mu$ whose $H_\mu$ is a diffeomorphism. 
\begin{prop}
Let $\mu$ be a probability measure. Then $H_{\mu}$ is a diffeomorphism on $\com+$
if and only if $\mu = \delta_a$ for some $a \in \real$.
\end{prop}
\begin{proof}
$\com+$ is analytically homeomorphic to the unit disc (denoted as $\Delta$) by 
the mapping $i\frac{z-i}{z+i}$. It is known that any bijective analytic map in $\Delta$ 
is of the form $\lambda \frac{z-b}{1+\bar{b}z}$ for some 
$\lambda \in \comp, |\lambda | = 1$ and $b \in \comp$, $|b| < 1$. Therefore, 
at least $H_\mu(z)$ takes the form as $\frac{a_1 z + a_2}{a_3z + a_4}$, where 
$a_k$'s are some complex numbers. Since $H_{\mu}$ is a reciprocal Cauchy transform, we have $a_3 = 0$ and $\frac{a_1}{a_4} = 1$ by Proposition \ref{maa} in \cite{Maa}. Thus $H_{\mu}(z) = z - a$ for some $a \in \real$.
\end{proof}

\section{Atoms in monotone convolution}\label{atom}
The monotone convolution of atomic measures appears in the monotone product of matrix algebras. 
It is easy to prove that the monotone convolution of $m \times m$ matrix and  $n \times n$ matrix becomes 
$mn \times mn$ matrix, which is a consequence of the algebraic construction of monotone product \cite{Mur3}.    
We study how atoms behave under monotone convolution:   
we prove an interlacing property of atoms in the monotone convolution of atomic measures. As a result, we 
obtain an interesting property which is not the case in the classical convolution (Corollary \ref{atom2}).

\begin{thm}\label{atom1}
$(1)$ Let $\nu := \sum_{k=1}^m \lambda _k \delta_{a_k}$ be an atomic probability measure such that 
$\lambda_k >0$, $\sum \lambda _k =1$ and $a_1 < a_2 < \cdots < a_m$.
For any $b \in \real$, $b \neq 0$, $\delta_b \rhd \nu$ has distinct $m$ atoms. 
When we write $\delta_b \rhd \nu = \sum_{k=1} ^m \mu_k \delta_{b_k}$ with $b_1 < \cdots < b_m$, 
the atoms satisfy either $b_1 < a_1 < b_2 < a_2 <\cdots < a_{m-1} < b_m < a_m$ 
or $ a_1 < b_1 < a_2 < b_2 < \cdots < a_m < b_m $.  
The coefficients $\mu_k$ are given by $\mu_i = \frac{\prod_{k=1}^m (b_i - a_k)}{b \prod_{k \neq i}^m (b_i - b_k)}.$  \\
$(2)$ Moreover, if $b$ and $c$ are distinct real numbers, 
the  $2m$ atoms appearing in $\nu_b = \delta_b \rhd \nu $ and $\nu_c = \delta_c \rhd \nu$
are all different.   
\end{thm}
\begin{rem}
Theorem \ref{atom1} shows a sharp difference between $\delta_b \rhd \nu$ and $\delta_b \ast \nu$: for instance, we can take $b > 0$ large enough so that the atoms $\{b_j \}$ of $\delta_b \ast \nu$ satisfy $a_1 < a_2 < \cdots < a_m < b_1 < b_2 < \cdots < b_m$, since $b_j = a_j + b$.  
\end{rem}

\begin{cor}\label{atom2}
Let $\mu$ be an atomic probability measure with distinct $m$ atoms and let $\nu$ be an atomic probability measure with distinct $n$ atoms. 
Then $\mu \rhd \nu$ consists of exactly distinct $mn$ atoms. 
\end{cor}
\begin{proof}[Proof of Theorem]
(1) The reciprocal Cauchy transform of $\delta_b \rhd \nu$ is 
\begin{equation}\label{nume}
H_{\delta_b \rhd \nu} (z) =  \frac{(z-a_1) \cdots (z-a_m) - b\sum_{j=1}^m \lambda _j \prod_{k=1,k \neq j}^m (z-a_k)}{\sum_{j=1}^m \lambda _j \prod_{k=1,k \neq j}^m (z-a_k)}. 
\end{equation}
Denote by $f(z)$ the  numerator of the right hand side of (\ref{nume}). Then we have
\begin{itemize}
\item[] $f(a_1) = - \lambda _1 b(a_1 - a_2)(a_1-a_3)\cdots(a_1-a_m) = (-1)^m b p_1 $, 
\item[] $f(a_2) = - \lambda _2 b(a_2 - a_1)(a_2-a_3)\cdots(a_2-a_m) = (-1)^{m-1} b p_2 $,  
\item[] ~~~~~~~~\vdots ~~~~~~~~~~~~~~~~~~~~~~~~~~~~~~~~~~~~~~~~~~~~~~~~~~~~~~~~~\vdots
\item[] $f(a_m) = - \lambda _m b(a_m - a_1)(a_m-a_2)\cdots(a_m-a_{m-1}) = - b p_m $,   
\end{itemize}
where $p_k$'s are some positive real numbers. The changes of signs of $f(z)$ and the behavior of $f(z)$ 
at $\infty$ and $-\infty$  show that there exist $m$ distinct real roots 
$b_1 < \cdots <  b_m$ of $f(z)$ as follows:

\begin{itemize}
\item[(a)]$b > 0$ 
$\Longrightarrow$ $b_k \in (a_k, a_{k+1})$ for $1\leq k\leq m-1,$ and $b_m \in (a_m,\infty)$.  
\item[(b)]$b < 0$ 
$\Longrightarrow$ $b_1 \in (-\infty, a_1)$ and $b_k \in (a_{k-1}, a_{k})$ for $2\leq k\leq m$.   
\end{itemize}

For the denominator, we look for $\mu_k$'s such that the following identity holds:
\begin{equation} \label{postu}
\sum_{j=1}^m \lambda _j \prod_{k=1,k \neq j}^m (z-a_k) = \sum_{j=1}^m \mu _j \prod_{k=1,k \neq j}^m (z-b_k).
\end{equation}     
These $\mu_k $'s are obtained as follows. When $z = b_i$, (\ref{postu}) becomes 
\begin{equation}
\mu_i = \frac{\sum_{j=1}^m \lambda _j \prod_{k=1,k \neq j}^m (b_i - a_k) }{\prod_{k \neq i}^m (b_i - b_k)}. 
\end{equation}
Conversely, if we define the $\mu_k$'s as above, the equality (\ref{postu}) holds at the different $m$ points 
$z = b_k$, $1\leq k \leq m$. 
Then the equality (\ref{postu}) holds identically since both sides of (\ref{postu}) are polynomials of at most degree $m-1$.
Thus we have obtained 
\begin{equation}
H_{\delta_b \rhd \nu}(z) = \frac{\prod_{k=1} ^m (z-b_k)}{\sum_{k=1}^m \mu_k \prod_{j \neq k, j=1}^m (z - b_j)}. 
\end{equation}
Since $f(z)$ is the numerator of $H_{\delta_b \rhd \nu}(z)$, it holds that 
$(b_i - a_1) \cdots (b_i - a_m) = b\sum_{j=1}^m \lambda _j \prod_{k=1,k \neq j}^m (b_i - a_k)$ 
for each $1 \leq i \leq m$.
Therefore, we obtain 
\begin{equation}
\mu_i = \frac{\prod_{k=1}^m (b_i - a_k)}{b \prod_{k \neq i}^m (b_i - b_k)}. 
\end{equation} 
Then we obtain $\delta_b \rhd \nu = \sum_{k = 1} ^m \mu_k \delta_{b_k}$. 

\noindent
(2) If $b$ or $c$ is equal to 0, the claim is obvious from (1). Hereafter, we consider the case $b \neq 0$ and 
$c \neq 0$.  
 In addition to $f(z)$ used in the proof of (1), we define $g(z)$ by 
\begin{equation} 
g(z) = (z-a_1) \cdots (z-a_m) - c\sum_{j=1}^m \lambda _j \prod_{k=1,k \neq j}^m (z-a_k).
\end{equation}   
Assume that there is some $\alpha$ which satisfies both $f(\alpha )=0 $ and $g(\alpha)=0$. 
Calculation of $f(\alpha ) - g(\alpha) = 0$ leads to
\begin{equation} \label{sb}
\sum_{j=1}^m \lambda _j \prod_{k=1,k \neq j}^m (\alpha - a_k) = 0, 
\end{equation} 
where $b \neq c$ has been used. Substituting (\ref{sb}) into the expression of $f(\alpha)=0$, 
we have 
\[(\alpha - a_1) \cdots (\alpha - a_m) = 0, \] 
which contradicts the fact that $\alpha$ is different from $a_k$'s. 
\end{proof}

We can characterize atomic probability measures in terms of the integral representation of reciprocal Cauchy transforms by a similar argument.  
\begin{prop}
A probability measure $\nu$ has the form 
$\sum_{k=1}^m \lambda _k \delta_{a_k}$ 
with $a_k < a_{k+1}$, $\lambda _k > 0$ for all $k$
if and only if 
its reciprocal Cauchy transform $H_\nu$ is of the form 
\[ H_\nu (z) = \alpha + z + \sum_{k=1}^{m-1} \frac{\beta_k}{b_k - z},  \]
with $\beta_k > 0$ and $\alpha \in \real$.
Moreover, if ${b_i} 's$ are ordered as $b_1 < b_2 < \cdots < b_{m-1}$ then 
it holds that $a_1 < b_1 < a_2 < b_2 < \cdots < b_{m-1} < a_m $.  
\end{prop}

For an atomic probability measure $\nu$ containing more than one atom, the number of atoms in $\nu ^{\rhd n}$ increases as $n$ increases  by Corollary \ref{atom2}. 
If we could prove that an $n$-th root of an atomic measure 
is again an atomic measure, then we could show that an atomic measure with finite atoms more than one is not monotone infinitely divisible by Corollary 
\ref{atom2}. We prove this fact next in a more general form without a reference to an $n$-th root. 

We say an atom $a$ in a probability measure $\mu$ is \textit{isolated} 
if $a \notin \overline{\supp \mu \backslash \{a \}}$.  
\begin{thm}\label{atom3} 
If a $\rhd$-infinitely divisible distribution $\nu$ contains an isolated atom at $a$,  
$\nu$ is of the form $\nu =  \nu(\{a \}) \delta_a + \nu_{ac}$, where $\nu_{ac}$ is absolutely continuous w.r.t. 
the Lebesgue measure and $a \notin \supp \nu_{ac}$. Moreover, we have 
\begin{equation}\label{emp}
\{u \in \supp \nu \backslash \{a \}; \limsup_{v \searrow 0}|G_{\nu}(u + iv)| = \infty \} = \emptyset. 
\end{equation}
\end{thm}

We need the following well-known fact, which is a consequence of the theorem of 
de la Vall\'{e}e Poussin \cite{Sak}. 
\begin{lem}\label{atom4}
For a positive finite measure $\nu$, the singular part $\nu_{sing}$ is supported on $\{u \in \supp \nu; |G_{\nu}(u + i0)| = \infty \}$. 
\end{lem}
\begin{proof}[Proof of Theorem] 
The probability measure $\nu$ is of the form $\nu =  \lambda \delta_{a} + \mu$, where $\lambda := \nu(\{a \}) >0$,
$\mu$ is a positive finite measure and $a \notin \supp \mu$.  
It is enough to prove that \\$\{u \in \supp \nu \backslash \{a \}; \limsup_{v \searrow 0}|G_{\nu}(u + iv)| = \infty \} = \emptyset$ by Lemma \ref{atom4}. 
We prove by \textit{reductio ad absurdum}. Assume that there exists a point $a_1$ such that $\limsup_{v \searrow 0} |G_{\nu_{ac}}(a_1 + iv)| = \infty$, which 
implies 
\begin{equation}\label{atom5}
\limsup_{v \searrow 0}H_{\nu}(a_1 + iv) = 0. 
\end{equation}  
It suffices to prove that $H_\nu$ is not injective on $\com+$ according to (c') explained in Section \ref{inj}. 
The reciprocal Cauchy transform of $\nu$ is given by 
\begin{equation} 
\begin{split} 
H_\nu (z) &= \frac{1}{\frac{\lambda}{z-a} + G_\mu(z)}  \\
          &= \frac{z-a}{\lambda + (z-a)G_\mu(z)}. 
\end{split}
\end{equation}
By the assumption $a \notin \supp\mu $, $G_\mu$ is analytic in 
some small neighborhood of $a$.

Let $z_1$ be an arbitrary point in $\com+$ and let $f(z)$ and $g(z)$ be analytic functions defined by  
\begin{gather}
f(z):= (z - a) -  H_\nu(z_1)\{  \lambda  +  (z - a)G_\mu (z) \}, \\ 
g(z):= (z - a). 
\end{gather}
We note that $f(z_2) = 0$ implies $H_{\nu}(z_1) = H_{\nu}(z_2)$. 
We shall prove that there exist a point $z_1 \in \com+$ and some small open disc $D$ around $a_2$ such that 
$|f(z) - g(z)| < |f(z)|$ on $\partial D$.   

We define $\eta :=\frac{1}{2} d(a, \supp \mu)$ and 
$D:= \{z \in \comp ;|z - a| < \eta   \}$,   
where $d(a, \supp \mu)$ is the distance between 
$a$ and $\supp \mu$.
Then $g(z)$ has just one zero point $a$ in $D$ and $D$ does not contain $z_1$ 
if $z_1$ is near to $a_1$. We have for $z \in \partial D$ 
\[|f(z)-g(z)| \leq M|H_\nu (z_1)|, \]
where $M$ is a constant independent of $z_1$. We also have for $z \in \partial D$ 
\begin{equation*}
\begin{split}
|f(z)| & \geq |z - a| - M|H_\nu (z_1)|  \\
       & \geq \frac{1}{2}\eta  - M|H_\nu(z_1)|.     \\  
\end{split}
\end{equation*}
If we take $z_1 = a_1 + yi$ with $y > 0$ to satisfy $M |H_\nu (z_1)| < \frac{1}{4} \eta $, then 
we have $|f(z)-g(z)| < |f(z)|$ on $\partial D$. 
Since $g(z)$ has only one zero point $a \in D$, $f(z)$ also has just one zero point $z_2$ in $D$ by Rouche's theorem. 
Then it follows that $H_{\nu } (z_1) = H_\nu (z_2)$ and $z_1 \neq z_2$. 
$\im z_2$ might be considered to be negative, which is, however, never the case. In fact, the reciprocal 
Cauchy transform $H_{\nu}$ defined on $\comp \backslash \supp \mu$ 
maps $\com+$ to $\com+$ and $\comp _{-}$ to $\comp _{-}$. Therefore, $\im z_2 > 0$, and the proof 
has been finished. 
\end{proof}     
\begin{rem}
There are $\rhd$-infinitely divisible probability distributions which contain one Dirac measure. 
For instance, a Dirac measure itself and the deformed arcsine law with parameter $c \geq 0$ \cite{Mur3} 
(see also Section \ref{Exa} of the present paper):
$d\mu_t = d\mu_{t,ac} + d\mu_{t,sing}$, where
\begin{equation}\label{abso1} 
\begin{split}
&d\mu_{t,ac}(x) =  \frac{1}{\pi }\frac{\sqrt{2t - (x-c)^2}}{c^2 + 2t - (x-c)^2} 1_{(c-\sqrt{2t}, c+ \sqrt{2t})}(x)dx, \\
&\mu_{t,sing} =  \frac{|c|}{\sqrt{c^2 + 2t}} \delta_{ c - \sqrt{c^2 + 2t} }.
\end{split}
\end{equation}  
\end{rem}
\begin{exa}
Let $0 < \lambda < 1$. The following examples do not satisfy (\ref{emp}). 
\begin{itemize}
\item[$(1)$] $\nu(dx) = \lambda \delta_a (dx) + \frac{1 - \lambda}{c-b} 1_{(b, c)}(x) dx$ with $a \notin (b,c)$ does not satisfy (\ref{emp}), since $G_{\nu_{ac}} (z) = \frac{1-\lambda}{c-b}\log\Big{(} \frac{z-b}{z-c} \Big{)}$.  
\item[$(2)$] $\nu(dx) = \lambda \delta_a (dx) + \frac{1-\lambda}{\pi \sqrt{2 - x^2}} 1_{(-\sqrt{2}, \sqrt{2})}(x) dx$ with $a \notin (-\sqrt{2}, \sqrt{2})$ does not satisfy (\ref{emp}) since 
 $G_{\nu_{ac}}(z) = \frac{1-\lambda}{\sqrt{z^2 - 2}}$.  
 \end{itemize}
More generally, we can prove under some restrictions that a point $u$ at which the density function is not continuous satisfies $|G_{\nu_{ac}}(u + i0)| = \infty$. 
We note that the deformed arcsine law $c \geq 0$ in (\ref{abso1}) has an atom if and only if $c > 0$, and the absolutely continuous part 
is a continuous function on $\real$ 
if and only if $c > 0$; there are no contradictions. 
\end{exa}

\section{Behavior of supports and moments under monotone convolution}\label{basic}
We consider properties of probability measures which are conserved under the monotone convolution. 
Let $\mu$ be a probability measure. Define the minimum and the maximum of the support: $a(\mu):= \inf \{x \in \text{supp}\mu \}$, $b(\mu):= \sup \{x \in \text{supp}\mu \}$.  Here  $-\infty \leq a(\mu) < \infty$ and $-\infty < b(\mu) \leq \infty$ hold. 
We say that $\mu$ contains an isolated atom at $c \in \real$ if $\mu(\{c \}) > 0$ and $c \notin \overline{(\supp \mu)\backslash \{c \} }$. 
In this paper we occasionally consider analytic continuations of functions such as $G_\mu$ or $H_\mu$ from $\comp \backslash \real$ to an open subset $U$ of $\comp$ which intersects $\real$. If there are no confusions, for simplicity, we only say that a function is analytic in $U$, instead of saying that a function has an analytic continuation. 
\begin{lem}\label{support}
Let $\mu$ be a probability measure. We use the notation (\ref{recip1}). \\
(1) $(\supp \mu)^c \cup (\comp \backslash \real)$ is the maximal domain in which $G_\mu(z)$ is analytic. Similarly,  
$(\supp \eta)^c \cup (\comp \backslash \real)$ is the maximal domain in which $H_\mu(z)$ is analytic. \\ 
(2) $\{x \in (\supp \mu)^c; G_\mu(x) \neq 0 \} \subset (\supp \eta)^c$. Similarly, 
 $\{x \in (\supp \eta)^c; H_\mu(x) \neq 0 \} \subset (\supp \mu)^c$. In particular, $a(\eta) \geq a(\mu)$ since $G_\mu(x) \neq 0$ for 
$x \in (-\infty, a(\mu))$. 
\end{lem}
\begin{proof}
These statements easily follow from the Perron-Stieltjes inversion formula. 
\end{proof} 

A classical infinitely divisible distribution necessarily has a noncompact support, except for a delta measure.
This situation is different from monotone, free and Boolean cases. For instance, 
a centered arcsine law is $\rhd$-infinitely divisible.   
The study of the maximum or minimum of a support becomes more important for this reason.
It is known that if $\lambda =\nu \rhd \mu$ and $\lambda $ has a compact support, then 
the support of $\mu$ is also compact \cite{Mur3}. We generalize this and prove a basic 
estimate of supports.  

\begin{prop}\label{estimate}
The following inequalities hold for probability measures $\nu $ and $\mu$. \\
(1) If $\supp \nu \cap (-\infty, 0] \neq \emptyset $ and 
$\supp \nu \cap [0, \infty) \neq \emptyset$, then  $a(\mu) \geq a(\nu \rhd \mu )$, $b(\mu) \leq b(\nu \rhd \mu )$.  \\
(2) If $\supp \nu \subset (-\infty, 0]$, then $a(\mu) \geq a(\nu \rhd \mu )$, $b(\nu) + b(\mu) \leq b(\nu \rhd \mu)$.  \\
(3) If $\supp \nu \subset [0, \infty)$, then $a(\nu) + a(\mu) \geq a(\nu \rhd \mu)$, $b(\mu) \leq b(\nu \rhd \mu) $.  
\end{prop}
\begin{proof} 
For a probability measure $\rho$, we denote by $\rho^x$ the probability measure $\delta_x \rhd \rho$. This is useful since 
$\nu \rhd \mu$ can be expressed as 
\begin{equation}\label{eq0111}
\nu \rhd \mu(B) = \int_{\real} \mu^x (B)\nu(dx)
\end{equation}
 for Borel sets $B$ \cite{Mur3}. 

Let $\lambda := \nu \rhd \mu$. We prove first the following inequalities for an arbitrary probability measure $\rho $: \\
\begin{equation*}
   \begin{cases}
             a(\rho ^x) \geq  a(\rho ),~~ b(\rho ^x) \leq b(\rho ) + x & \text{for all $x > 0$}, \\
             a(\rho ^x) \geq  a(\rho ) - |x|,~~ b(\rho ^x) \leq  b(\rho ) & \text{for all $x < 0$.} 
   \end{cases}
\end{equation*}
It easy to prove that $\rho^x$ can be characterized by $G_{\rho ^x} = \frac{G_\rho }{1-xG_\rho }$. 
If $x > 0$, then $1-xG_{\rho}(z) \neq 0 $ for $z \in \comp \backslash [a(\rho), b(\rho) + x]$
and $G_{\rho}$ is analytic in this domain.
Therefore, the first inequality holds. The second is proved similarly.    

Let $J:= \supp \lambda$.
In view of the relation $\lambda (A) = \int_{\real} \mu^x(A)d\nu(x)$, 
we have $\lambda (J^c) = \int_{\real} \mu^x(J^c)d\nu(x) = 0$. 
Hence we obtain $\mu^x(J^c) = 0$, $\nu$-a.e. $x \in \real$. 
Take any $x_0$ such that $\mu^{x_0}(J^c) = 0$. Then  we have $a(\mu^{x_0}) \geq a(\lambda )$ and 
$b(\mu^{x_0}) \leq b(\lambda )$. 
If $x_0 > 0$, combining the inequalities $a(\rho ^x) \geq a(\rho ) - |x|$ and  $b(\rho ^x) \leq b(\rho ) $ for 
$\rho = \mu^{x_0}$ and $x =  - x_0 < 0 $, we have 
\begin{align*}
&a(\mu) = a(\mu^{x_0 - x_0}) \geq a(\lambda )  - |x_0|, \\ 
&b(\mu) = b(\mu^{x_0 -x_0}) \leq b(\lambda ). 
\end{align*}
Similarly if $x_0 < 0$,  
\begin{align*} 
&a(\mu) \geq  a(\lambda ),  \\
&b(\mu) \leq  b(\lambda ) + |x_0|. 
\end{align*}
Assume that $ \supp \nu \subset  (-\infty, 0]$. 
Then we obtain  $a(\mu) \geq a(\lambda )$ and 
 $b(\mu) \leq b(\lambda ) + |b(\nu)|$  
since there is a sequence of such $x_0$'s converging to the point $b(\nu)$. 
Hence we have proved (2). The statements (1) and (3) are proved in a similar way to (2). 
\end{proof} 
\begin{cor}\label{positivity} Let $\nu$ be a probability measure and let $n \geq 1$ be a natural number. \\
$(1)$  If $\supp (\nu ^{\rhd n} )\subset (- \infty, 0]$, then $\supp \nu \subset (- \infty, 0]$ and $|b(\nu)| \geq \frac{1}{n}|b(\nu ^{\rhd n} )|$. \\
$(2)$ If $\supp (\nu^{\rhd n}) \subset [0, \infty)$, then $\supp \nu \subset [0, \infty)$ and $a(\nu) \geq \frac{1}{n}a(\nu^{\rhd n})$. 
\end{cor}
This corollary puts a restriction on the support of a $\rhd$-infinitely divisible distribution.
The continuous time version of (2) will be proved in Section \ref{time}. 
\begin{proof}
Let $\lambda:=\nu^{\rhd n}$. \\
(1) Assume that both $b(\nu) > 0$ and $b(\lambda ) = b(\nu ^{\rhd n}) \leq   0$ hold, then 
there are two possible cases:  (a) $\supp \nu \cap [0, \infty) \neq \emptyset$ 
and $ \supp \nu \cap (-\infty, 0] \neq \emptyset$; (b) $\supp \nu \subset [0, \infty)$ 
in Proposition \ref{estimate}. We apply Proposition \ref{estimate} replacing $\lambda$ and $\mu $ with $\nu^{\rhd n}$ and $\nu^{\rhd n-1}$, 
respectively. In both cases (a) and (b), it holds that $b(\nu ^{\rhd n-1}) \leq b(\lambda) \leq 0$.
Thus we obtain $b(\nu^{\rhd n-1}) \leq 0$. This argument can be repeated and finally we have $b(\nu) \leq 0$, a contradiction. 
Thereofre, $b(\nu) \leq 0$.  By the iterative use of Proposition \ref{estimate} (2) we obtain $b(\nu ^{\rhd n}) \geq nb(\nu)$,  
from which the conclusion follows. 
A similar argument applies to (2).
\end{proof}


The following theorem is well known. We will need almost the same argument in Proposition \ref{positive support2}. 
\begin{lem}\label{position of atom}
For a finite measure $\mu$, $\lim_{y \searrow  0}iyG_{\mu}(a + iy) = \mu(\{a \})$ for all $a \in \real$.
\end{lem}
\begin{proof}
This claim follows from the dominated convergence theorem. 
\end{proof}

Now we prove a condition for a support to be included in the positive real line. 
A similar result was obtained in \cite{Ber1}. 
\begin{prop}\label{positive support2}
We use the notation (\ref{recip1}). 
Then $\supp \mu  \subset [0, \infty)$ if and only if $\supp \eta \subset [0, \infty)$ and $H_\mu (-0) \leq 0$ hold. 
Moreover, under the condition $\supp \eta \subset [0, \infty)$, the condition $H_\mu (-0) \leq 0$ is equivalent to the following conditions: $(\ast)$ $\eta(\{0 \})=0$; 
$\int_0 ^{\infty} \frac{1}{x}d\eta(x) < \infty$; $b + \int_0 ^{\infty} \frac{1}{x}d\eta(x) \leq 0$. 
\end{prop}
\begin{proof} 
If $\supp \eta \subset[0, \infty)$ and $H_\mu (-0) \leq 0$, we have $H_\mu (u) < 0$ for all $u < 0$ since $H_\mu$ is strictly increasing. Then $G_\mu = \frac{1}{H_\mu}$ is analytic in $\comp \backslash [0, \infty)$, 
which implies $\supp \mu \subset [0, \infty)$. 
Conversely, we assume $\supp \mu  \subset [0, \infty)$. 
By Lemma \ref{support}, we have $\supp \eta \subset [0,\infty)$.  If $H_\mu (-0)$ were greater than $0$, there would exist $u_0 < 0$ such that $H_\mu(u_0) = 0$. 
Then $\mu$ has an atom at $u_0 < 0$, which contradicts the assumption. Therefore, $H_\mu (-0) \leq 0$. 

We show the equivalence in the last claim. It is not difficult to prove that $(\ast)$ implies $H_\mu (-0) \leq 0$. 
Now we shall prove the converse statement.  Assume that $\lambda := \eta (\{0 \}) > 0$. By a similar argument to Lemma \ref{position of atom}, we can prove that $\lim_{u \nearrow 0} uH_\mu(u) = - \lambda$. Therefore, for $u < 0$ sufficiently close to $0$, we have $H_\mu(u) > -\frac{\lambda}{2u} > 0$, which contradicts the condition $H_\mu (-0) \leq 0$.  Then we have $\eta(\{0 \}) = 0$. 
Since $f_u(x) := \frac{1 + xu}{x - u}$ is increasing with respect to $u$, 
we can apply the monotone convergence theorem and obtain the two inequalities $\int_0 ^{\infty} \frac{1}{x}d\eta(x) < \infty$ and 
$b + \int_0 ^{\infty} \frac{1}{x}d\eta(x) \leq 0$. 
\end{proof}

\begin{cor}
The monotone convolution preserves the set $\{\mu; \supp \mu \subset [0, \infty) \}$ of probability measures. 
\end{cor}
\begin{proof} 
If $\supp \mu \subset [0, \infty)$ and $\supp \nu \subset [0, \infty)$, $H_{\mu \rhd \nu} = H_\mu \circ H_\nu$ is analytic in $\comp \backslash [0, \infty)$. 
Since $H_{\mu \rhd \nu}$ is increasing in $(-\infty, 0)$, we have $H_{\mu \rhd \nu}(-0) = H_{\mu} \circ H_\nu(-0) \leq H_{\mu}(-0) \leq 0$. 
By Proposition \ref{positive support2}, we obtain $\supp (\mu \rhd \nu) \subset [0, \infty)$. 
\end{proof}

\begin{rem}
The above property is also true for Boolean convolution. The proof goes similarly. We note that the corollary follows immediately if we use the operator-theoretic realization of monotone independent random variables in \cite{Fra3}. 
\end{rem}

Next we consider moments. Let $m_n (\mu):=\int_{\real} x^n \mu(dx)$ be the $n$-th moment of a probablility measure $\mu$.  
\begin{prop}\label{asymptotic2}
Let $\mu$ be a probability measure and let $n \geq 1$ be a natural number. Then the following conditions are equivalent. 
\begin{itemize}
\item[(1)] $m_{2n}(\mu) < \infty$, 
\item[(2)] $H_\mu$ has the expression 
$H_\mu(z) = z + a + \int_{\real}\frac{\rho(dx)}{x-z}$, where $a \in \real$ and $\rho$ is a positive finite measure satisfying $m_{2n-2}(\rho) < \infty$,  
\item[(3)] there exist $a_1, \cdots, a_{2n} \in \real$ such that 
\begin{equation}\label{expansion2} 
H_{\mu}(z) = z + a_1 + \frac{a_2}{z} + \cdots + \frac{a_{2n}}{z^{2n-1}} + o(|z|^{-(2n - 1)})  
\end{equation} 
for $z = iy ~(y \to \infty)$. 
\end{itemize}
If (3) holds, for any $\delta > 0$ the expansion (\ref{expansion2}) holds for $z \to \infty$ satisfying $\im z > \delta |\re z|$. 
Moreover, we have $a_{k+2} = -m_k(\rho)$ $(0 \leq k \leq 2n-2)$. 
\end{prop}
\begin{proof}
The equivalence $(1) \Leftrightarrow (3)$ follows from Theorem 3.2.1 in \cite{Akh} by calculating the reciprocals. The implication 
$(2) \Rightarrow (3)$ is not difficult. The proof of $(3) \Rightarrow (2)$ runs by the same technique as in Theorem 3.2.1 in the book \cite{Akh}. 
\end{proof}
\begin{prop}\label{moment21}
Let $\mu$ and $\nu$ be probability measures and let $n \geq 1$ be a natural number. 
If $m_{2n}(\mu) < \infty$ and $m_{2n}(\nu) < \infty$, then $m_{2n}(\mu \rhd \nu) < \infty$. Moreover, we have 
\begin{equation}\label{moment20}
m_l(\mu \rhd \nu) = m_l(\mu) + m_l(\nu) + \sum_{k = 1} ^{l-1} \sum_{\substack{j_0 + j_1 + \cdots +j_k = l - k, \\  0 \leq j_p, ~0 \leq p \leq k }} m_k(\mu) m_{j_0}(\nu)\cdots m_{j_k}(\nu)  
\end{equation}
for $1 \leq l \leq 2n$. 
\end{prop}
\begin{proof}
We note that $\im H_\nu (z) \geq \im z$.  
For any $\delta > 0$, there exists $M = M(\delta) > 0$ such that 
\begin{equation}\label{a19}
\im H_\nu(iy) \geq y > \delta |\re H_\nu(iy)| \text{~for~} y > M. 
\end{equation} 
By (\ref{expansion2}), we obtain 
\begin{equation}
H_\mu(H_\nu(iy)) = H_\nu (iy) + a_1 + a_2G_\nu(iy) + \cdots + a_{2n}G_\nu(iy)^{2n-1} + R(H_\nu(iy)),    
\end{equation}
where $z^{2n-1}R(z) = \int_{\real} \frac{x^{2n-1}}{x - z}\rho(dx) \to 0$ as $z \to \infty$ satisfying $\im z > \delta|\re z|$ for a fixed $\delta > 0$. We have 
\[
y^{2n-1} |R(H_\nu(iy))| \leq |H_\nu(iy)|^{2n-1}|R(H_\nu(iy))| \to 0 
\] as $y \to \infty$ by the condition (\ref{a19}). Thus $R(H_\nu(iy)) = o(y^{-(2n-1)})$. Expanding $H_\nu(z)$ 
in the form (\ref{expansion2}), we can see that 
there exist $c_1, \cdots, c_{2n} \in \real$ such that $H_\mu(H_\nu(z)) = z + c_1 + \frac{c_2}{z} + \cdots + \frac{c_{2n}}{z^{2n-1}} + o(|z|^{-(2n-1)})$ for $z = iy$ $(y \to \infty)$. Then the $2n$-th moment of $\mu \rhd \nu$ is finite by Proposition \ref{asymptotic2}. The equality (\ref{moment20}) is obtained by the expansion of 
$G_{\mu \rhd \nu}(z) = G_\mu (\frac{1}{G_\nu(z)})$. 
\end{proof}

\section{Differential equations arising from monotone convolution semigroups}\label{Diff}
Let $\{\mu_t \}_{t \geq 0}$ be a weakly continuous $\rhd$-convolution semigroup with $\mu_0 = \delta_0$. 
We denote $H_{\mu_t}$ by $H_t$ for simplicity. We sometimes write $H(t, z)$ to express explicitly 
that $H_t(z)$ is a function of two variables. 
By (\ref{recip1}), $H_t$ can be expressed as 
\begin{equation}
H_t (z) = b_t + z + \int_{\real} \frac{1 + xz}{x-z}d\eta_t(x), 
\end{equation}
where, for each $t > 0$, $a_t$ is a real number and $\eta_t$ is a finite positive measure.  
We denote by $A(z)$ the associated vector field throughout this paper.

Throughout this section, we will prove the following properties of the minimum of the support of a convolution semigroup. 
\begin{thm}\label{Diffeq}
Let $\{\mu_t \}_{t \geq 0}$ be a weakly continuous $\rhd$-convolution semigroup 
with $\mu_0 = \delta_0$. We assume that for every $t>0$ $\mu_t$ is not a delta measure. We have such a form  
$\mu_t = \lambda (t)\delta_{\theta (t)} + \nu_t$ with 
$\theta(t) \notin \supp \nu_t$, $\theta(t) = a(\mu_t)$ and $\lambda(t)\geq 0$. \\
$(1)$ Assume $a(\tau) > 0$. Then there are four cases: 
\begin{itemize}
\item[$(A)$]If $A(u_0)$ = 0 for some $u_0 \in [-\infty, 0)$ and $A(u) < 0$ on $(-\infty, u_0)$ 
and $A(u) > 0$ on $(u_0, 0)$ $($when $u_0 = -\infty$, we understand the condition as $A > 0)$, 
then $\lambda (t)  > 0$. 
Moreover, the inequality $u_0 < \theta (t) < 0$ holds for all $t > 0$.  
\item[$(B)$]If $A(u) < 0$ on $(-\infty, 0)$ and $A(0) = 0$, then 
$\theta(t) = 0$ and $\lambda (t) > 0$ for all $t > 0$. 
\item[$(C)$]If there exists $u_0 \in (0, a(\tau))$ such that $A(u) < 0$ on $(-\infty, u_0)$ and 
$A(u) > 0$ on $(u_0, a(\tau))$, then it follows that  
$\theta(t) \in (0, u_0)$ and $\lambda (t) > 0$ for $0 < t < \infty$ and 
$\lambda (t) > 0 $ for $ t > 0$.
\item[$(D)$]If $A(u) < 0$ on $(-\infty, a(\tau))$, then there exists $t_0 \in(0, \infty]$ 
such that $\lambda (t) > 0$ for all $0 < t < t_0 $ and $ \lambda (t)= 0$ for $t_0 \leq t < \infty$. 
\end{itemize} 
If $A(0) \neq 0$ and $\lambda (t) > 0$, the weight of the delta measure is written as  
$\lambda (t) = \frac{A(\theta(t))}{A(0)}$. If $A(0) = 0$ $(case (B))$, then we have $\lambda (t) = e^{-A'(0)t}$. 
Concerning the position of the delta measure, the following ODE holds: 
\begin{equation}
\begin{cases} \frac{d}{dt}\theta (t) = -A(\theta (t))$,$ \label{ODE of theta2} \\ 
              \theta (0) = 0. 
\end{cases}
\end{equation}

\noindent
$(2)$ We assume $a(\tau) > -\infty$. There are three cases in terms of the signs of the associated vector field: 
\begin{itemize}
\item[$(a)$] $A(u) > 0$ on $(-\infty, a(\tau));$  
\item[$(b)$] $A(u_0) = 0$ for some $u_0 \in (-\infty, a(\tau))$ 
and $A(u) < 0$ on ($-\infty, u_0)$ and   $A(u) > 0$ on $(u_0, a(\tau));$   
\item[$(c)$] $A(u) < 0$ on $(-\infty, a(\tau))$.  
\end{itemize}
In case $(a)$ and case $(b)$, we have the following ODE for $a(\nu_t)$: \\ 
\begin{equation}
\begin{cases}  \frac{d}{dt} a(\nu_t) = -A(a(\nu_t)), \\
                a(\nu_0) = a(\tau).  \label{ODE of min}\\
\end{cases}
\end{equation}
In case $(c)$, the equality $a(\nu_t) = a(\tau)$ holds for a.e. $t$ and $a(\nu_t) \geq a(\tau)$ 
for all $t \in [0, \infty)$. 
Moreover, if $\lim_{u \nearrow a(\tau)} A(u) < 0$, we have $a(\nu_t) = a(\tau)$ for all $t$. 
\end{thm} 

\begin{exa}\normalfont We can confirm the validity of the ODEs of $\theta(t)$ and $a(\nu_t)$, and the validity of the formula of the
weight of a delta measure in each example. 
\begin{itemize} 
\item[$\cdot$] Arcsine law: $\mu_t = \frac{1}{\pi \sqrt{2t - x^2}} 1_{(-\sqrt{2t}, \sqrt{2t})}(x)dx$, $A(z)= -\frac{1}{z}$, $a(\tau) = 0$,  $a(\mu_t) = -\sqrt{2t}$. 
\item[$\cdot$] A deformation of $\alpha$-strictly stable distributions ($0 < \alpha < 2$) with parameter $c \in \comp$, $\im c = 0, \re c \geq 0$ (see \cite{Has2}): $\mu_t = \mu_{t, ac},  ~\supp \mu_{t, ac} = (-\infty, c + t^{\frac{1}{\alpha} }]$,  
$A(z) = - \frac{1}{\alpha}(z-c)^{1-\alpha}$.  
We can check that the solution of the ODE (\ref{ODE of min}) is $c + t^{\frac{1}{\alpha}}$ (the same ODE (\ref{ODE of min}) holds for $b(\mu_t)$). 
\item[$\cdot$] The monotone Poisson distribution with parameter $\lambda > 0$: 
$\mu_t (dx) = \mu_{t,ac} + \mu_{t,sing}$, $A(z) = \frac{\lambda z}{1-z}$, where $\mu_{t,sing}$ is a delta measure at $0$. 
, and hence, 
it holds that $A(0) = 0$ and $A'(0)= \lambda$. This is the case (B). Therefore, we have $\mu_{t,sing}=e^{-\lambda t}\delta_0$.
\end{itemize} 
\end{exa}


\subsection{Differential equation of delta measure}\label{delta}
We summarize three equalities, some of which were used by Muraki in \cite{Mur3}.
\begin{lem}\label{equalities}
Let $\{\mu_t \}_{t \geq 0}$ be a weakly continuous $\rhd$-convolution semigroup 
with $\mu_0 = \delta_0$. Then we have three equalities on $\comp \backslash \real$: \\
(1) $A(H_t(z)) = A(z)\frac{\partial H_t}{\partial z}(z)$; \\
(2) $\frac{\partial }{\partial t} G_t(z) = A(z)\frac{\partial }{\partial z}G_t(z)$;  \\ 
(3) $\frac{\partial }{\partial t} H_t(z) = A(z)\frac{\partial }{\partial z}H_t(z)$. 
\end{lem}
\begin{proof}
Since $H_t(z)$ is a flow in $\comp \backslash \real$, $H_t \circ H_s = H_{t+s}$ for $t,s \geq 0$. 
(1) follows from the derivative $\frac{\partial }{\partial s}|_{s=0}$. 
(3) follows from (1) and (\ref{ODE}). (2) follows from (3) immediately. 
\end{proof}

First we treat a distribution which contains a delta measure at the minimum of the support. 
Suppose that $\{\mu_t \}_{t \geq 0}$ is a weakly continuous $\rhd$-convolution semigroup 
with $\mu_0 = \delta_0$. Then $\mu$ can be written as $\mu = \lambda \delta_\theta + \nu$ with $\theta \in (\supp \nu )^c$ and $0 < \lambda < 1 $. 
We use the integral representation in Theorem \ref{inf.div} (4) for the associated vector field $A(z)$. 
Throughout this subsection, we assume that $A$ is not a real constant which means that $\mu_t$ is not a delta measure for any $t > 0$ and that $a(\tau) > 0$. We shall show that there exists a delta measure at 
the minimum point of the support for some (finite or infinite) time interval. 
Moreover, the weight of a delta measure is calculated.

The derivative of $A$ satisfies $A'(u) > 0$ for all $u \in (-\infty,0)$.  
This implies that there are five possible cases: 
\begin{itemize}
\item[(A)] $A(u) > 0$ on $(-\infty, 0)$;  
\item[(A')] $A(u_0)$ = 0 for some $u_0 \in (-\infty, 0)$ and $A(u) < 0$ on ($-\infty, u_0)$ and 
  $A(u) > 0$ on $(u_0, 0)$;
\item[(B)] $A(u) < 0$ on $(-\infty, 0)$ and $A(0) = 0$;
\item[(C)]there exists $u_0 \in (0, a(\tau))$ such that $A(u) < 0$ on $(-\infty, u_0)$ and 
$A(u) > 0$ on $(u_0, a(\tau))$;  
\item[(D)] $A(u) < 0$ on $(-\infty, a(\tau))$;
\end{itemize}   

We consider the solution of the ODE (\ref{ODE}) also on the real line as well as on $\comp \backslash \real$.

\vspace{10pt}
\noindent
\textbf{\underline{Case (A) and case (A')}} \\ 
Case (A) is reduced to case (A') if we define $u_0:= -\infty$. Since $H(t, u)$ is an increasing function of $u  \in (\supp \eta_t)^c$, 
there is a unique point $\theta(t)$ satisfying $u_0 < \theta (t) < 0$ and 
\begin{equation}
H(t, \theta (t)) = 0 \label{zero of H}.
\end{equation} 
$\theta(t)$ is a zero point of $H_t$ of degree 1 since $\partial_u H(t, u) \geqq 1$. 
Therefore, by lemma \ref{position of atom},  there is a delta measure $\lambda (t) \delta_{\theta (t)}$ in $\mu_t$ with $u_0 < \theta (t) < 0$. By the implicit function theorem, $\theta (t)$ is in $C^{\omega}$ class.
Differentiating the equation $H(t, \theta(t)) = 0$ and using Lemma \ref{equalities}, we obtain 
\begin{equation}\label{ODE of theta}
\theta '(t) = -\frac{  \frac{\partial H}{\partial t}(t, \theta (t))   }{  \frac{\partial H}{\partial z}(t, \theta (t))  }=- A(\theta (t)). 
\end{equation} 
The initial condition is $\theta(0)=0$. 

\vspace{10pt}
\noindent
\textbf{\underline{Case (B)}} \\ 
In case (B), the same differential equation (\ref{ODE of theta}) holds. 
Since $A(0) = 0$, we have $\theta (t) = 0$ for all $t$. This is true 
for a monotone Poisson distribution.

\vspace{10pt}
\noindent
\textbf{\underline{Case (C) and case  (D)}} \\
Case (C) and case (D) can be treated at the same time. We define  
\[
u_1 := \begin{cases}
                 u_0,~& \text{in case (C)}, \\
                 a(\tau),~ & \text{in case (D)}, 
       \end{cases} 
\]
to treat the two cases at the same time. 
In the cases (C) and (D), $H_t$ is analytic in $\comp \backslash [u_1, \infty)$  
(see Subsection \ref{non-atomic} for details). 
Then there exists $t_0 \in (0, \infty]$ such that $\mu_t$ includes a delta measure in $(0, u_1)$ 
for $0 < t < t_0$. We can prove that $t_0 = \infty$ in case (C). In case (D), we have an example, where 
$t_0 < \infty$ holds (see the section of Example in \cite{Has2}). $t_0 = \infty$ may occur if $\lim_{u \nearrow a(\tau)} A(u) = 0$. 
$\mu_t$ has the form 
\[
\mu_t = \begin{cases} 
                    \lambda (t) \delta_{\theta (t)} + \nu_t,~&  0 \leq t < t_0,  \\ 
                    \nu_t,~&  t_0 \leq t < \infty, \\  
\end{cases}
\] 
where it holds that $0 < \lambda (t) \leq 1$ and $0 \leq \theta (t) < a(\tau)$ for $0 \leq t < t_0$, 
and $a(\nu_t) \geq a(\tau)$ for all $0 < t < \infty$. The differential equation (\ref{ODE of theta}) holds also in this case.

\vspace{10pt}
\noindent
\textbf{\underline{Weight $\lambda(t)$ in the cases (A), (A'), (C) and (D)}} \\
It is possible to calculate the weight $\lambda (t)$. First we exclude case (B). 
Then we have $A(0) \neq 0$. 
We expand $H_t(z)$ in a Taylor series around $\theta (t)$ as $H_t (z)= \sum_{n = 1} ^{\infty} a_n (t) (z - \theta (t))^n$  
with $a_1 (t) = \frac{1}{\lambda (t)}$. 
Also we expand $A(z)$ as $\sum_{n = 0} ^{\infty} b_n z^n$ with $b_n \in \real$. 
If we compare the coefficients of the constant term in the ODE (\ref{ODE}), 
we obtain $-\theta '(t)a_1 (t) = b_0 = A(0)$. Hence it holds that 
\[ \lambda (t) = \frac{A(\theta (t))}{A(0)}. \]

\vspace{10pt}
\noindent
\textbf{\underline{Weight $\lambda(t)$ in the case (B)}} \\
In case (B), we express the Taylor expansions of $H_t$ and $A(z)$ at $0$ respectively by $H_t (z) = \sum_{n=1} ^{\infty}a_n (t)z^n$ and 
$A(z) = \sum_{n=1}^{\infty}b_n z^n$ with $a_1(t) = \frac{1}{\lambda (t)}$ and $b_1 = A'(0) > 0$. 
Comparing the coefficients of $z^n$ in the ODE (\ref{ODE}), we obtain the equation $a'_1(t) = A'(0)a_1(t)$. 
Therefore, we get $a_1 (t) = e^{A'(0)t}$ because of the initial condition $a_1 (0) = 1$. Thus we obtain 
\[ \lambda (t) = e^{-A'(0)t}. \]

\subsection{Differential equation of non-atomic part}\label{non-atomic} 
In the previous subsection we considered the case $a(\tau) > 0$. Now we consider a more general case. 
We investigate $a(\mu_t)$ including the case where there is no isolated delta measure at $a(\mu_t)$.
Assume that the lower bound $a(\tau)$ of the L\'{e}vy measure $\tau$ is finite: $-\infty < a(\tau)$. There are three cases:
\begin{itemize}
\item[(a)] $A(u) > 0$ on $(-\infty, a(\tau))$;  
\item[(b)] $A(u_0)$ = 0 for some $u_0 \in (-\infty, a(\tau))$ 
and $A(u) < 0$ on ($-\infty, u_0)$ and   $A(u) > 0$ on $(u_0, a(\tau))$;   
\item[(c)] $A(u) < 0$ on $(-\infty, a(\tau))$.  
\end{itemize}
$\mu_t$ may contain an isolated delta measure at $a(\mu_t)$. If so, we write as 
$\mu_t = \lambda (t)\delta_{\theta (t)} + \nu_t$. We can understand that $\lambda(t) = 0$ if $\mu_t$ does not contain an atom at $a(\mu_t)$, or if 
$\mu_t$ contains an atom at $a(\mu_t)$ but it is not isolated. 
The motion of the position $\theta (t)$ of a delta measure was clarified in the previous subsection. To investigate $a(\nu_t)$, 
we introduce a function $E$: $[0, \infty) \longrightarrow (-\infty, a(\tau)]$ by 
\[ 
E(t):= \begin{cases} 
      \sup \{u \leq a(\tau); H_t (u) = a(\tau) \} &  \text{~in case (a) and case (b), } \\
      a(\tau) & \text{~in case (c)} 
       \end{cases}
\] 
for $t \in [0, \infty )$. The definition in the cases (a) and (b) may seem to be unclear since $H_t(z)$ was only defined in $\comp \backslash \real$. 
The precise definition is as follows. Since case (a) and case (b) can be treated in the same way, we explain only case (b). 
If $u$ is in the interval $(u_0, a(\tau))$, let $R(u)$ be 
defined so that $H_t(u)$ exists for all $t \in (0, R(u))$ and $\lim_{t \nearrow R(u)}H_t (u) = a(\tau)$. 
We observe that $R$ is a function of $u$ which satisfies $0 < R(u) < \infty$ on $(u_0, a(\tau))$. 
$R$ is a bijection from $(u_0, a(\tau))$ to $(0, \infty)$. 
Therefore, we can define a bijection $E(t):=R^{-1}(t)$, which we have denoted simply as 
$\sup \{u \leq a(\tau); H_t (u) = a(\tau) \}$.

$a(\nu_t)$ is characterized by the following result.   
\begin{lem} Let $\mu$ be a $\rhd$-infinitely divisible distribution. $\mu$ can be expressed in the form $\mu = \lambda \delta_\theta + \nu$, where $\theta = a(\mu)$ is an isolated atom. We understand that $\mu = \nu$ or $\lambda = 0$ if $\mu$ does not contain an atom at $a(\mu)$ or if 
$\mu$ contains an atom at $a(\mu)$ but it is not isolated. Then the equalities 
\[
a(\nu) = a(\eta) = \sup \{x \in \real; H_{\mu} \text{ has an analytic continuation to } \comp \backslash [x, \infty) \} 
\]
hold under the notation (\ref{recip1}).
\end{lem}
\begin{proof} The latter equality follows from Lemma \ref{support} (1) immediately and we only need to prove that $a(\nu) = a(\eta)$.  
First, if $\lambda =0$ we can easily prove $a(\mu) = a(\eta)$ by Lemma \ref{support} (2). 
Second, we assume that $\lambda > 0$. We show that $a(\nu) \neq a(\eta)$ 
causes a contradiction. We notice first that the difference $a(\nu) \neq a(\eta)$ comes from the zero points of $H_\mu(x)$ or $G_\mu(x)$ by Lemma \ref{support} (2). If $a(\nu) < a(\eta)$, then $H_\nu(a(\nu)) = 0$. This implies, however, $G_\mu$ contains two atoms at $a(\nu)$ and $\theta$. 
This contradicts infinite divisibility (see Theorem 3.5 in \cite{Has2}). If $a(\nu) > a(\eta)$, then $G_{\nu}(a(\eta)) = 0$. Since $\frac{d}{dx}H_\mu(x) \geq 1$ in $(\supp \mu)^c \subset \real$, $H_\mu(x)$ is increasing. Therefore $\lim_{x \nearrow a(\eta)}H_\mu(x) = \infty$ and $\lim_{x \searrow a(\eta)}H_\mu(x) = -\infty$. Also, $\lim_{x \to -\infty}H_\mu(x) = -\infty$. These imply that there exist $x_1 < a(\eta)$ and $x_2 > a(\eta)$ such that 
$H_\mu(x_1) = H_\mu(x_2)$. By Rouche's theorem, there exist distinct points $z_1,z_2 \in \comp$ with positive imaginary parts such that $H_\mu(z_1) = H_\mu(z_2)$ (this argument is similar to the proof of Theorem 3.5 in \cite{Has2}); this contradicts the infinite divisibility again since the solution of (\ref{ODE}) defines a flow of 
injective mappings.   
\end{proof} 
\begin{rem}
If $\mu$ is not $\rhd$-infinitely divisible, the above property does not hold. For instance, if $\mu=\frac{1}{2}(\delta_{-1} + \delta_1)$, 
$a(\nu) = 1$ but $a(\eta) = 0$. 
\end{rem}
We define $a(\nu_0):= a(\tau)$ in order that $a(\nu_t)$ becomes a continuous function around $0$. 
 
\begin{thm}\label{E1}
In case $(a)$ and case $(b)$, the equality $E(t) = a(\nu_t)$ holds for all $t \in [0, \infty)$. 
In case $(c)$, the equality holds under the further assumption $\lim_{u \nearrow a(\tau)} A(u) < 0$.  
\end{thm} 
\begin{proof}
We can prove this equality by considering the region in which $H_{t} (z)$ is analytic. 
We first consider case (a) and case (b).  We prove that  
\begin{equation}\label{eqE}
E(t) = \sup \{x \in \real; H_t \text{ has an analytic continuation to } \comp \backslash [x, \infty) \}.
\end{equation}
By \textit{reductio ad absurdum} we show that $H_t$ never has an analytic continuation beyond $E(t)$.
If $H_t(z)$ has an analytic continuation to 
$\comp \backslash [E(t) + \delta, \infty)$ 
for some $t > 0$ and $\delta > 0$, then we find the following three facts: 
the image of $H_t (u)$ includes the point $a(\tau)$ since 
$\frac{\partial H}{\partial u} \geq 1$ and $H(t, E(t)) = a(\tau)$; 
$H_t$ is injective in $\comp \backslash [E(t) + \delta, \infty )$; 
 we can take $\delta > 0$ small enough so that 
$A(z)$ is analytic in $\comp \backslash [E(t) + \delta, \infty )$ since $E(t) < a(\tau)$. 
Then by the equality $A(H_t(z)) = A(z)\frac{\partial H_t}{\partial z}(z)$ in $\comp \backslash \real$, we conclude that 
$A(z)$ has an analytic continuation to the image of $H_t$. 
In particular, $A$ is analytic around the point $a(\tau)$; this is a contradiction. Therefore, $H_t$ cannot have an analytic continuation beyond $E(t)$. 
    
Conversely, for any $u < E(t)$, $H_{t}(z)$ has an analytic continuation 
to the region $\comp \backslash [u + \delta, \infty )$ for some $\delta > 0$ 
by the solution of the ODE (\ref{ODE}). Then the equality (\ref{eqE}) holds. 

The proof of the equality $E(t) = a(\nu_t)$ in case (c) under the assumption $\lim_{u \nearrow a(\tau)} A(u) < 0$ is similar to the above. 
For all $t > 0$, we have $\lim_{u \nearrow a(\tau)}H_t (u) < a(\tau)$. 
Assume that $H_t(z)$ has an analytic continuation to 
$\comp \backslash [E(t) + \delta, \infty)$ 
for some $t > 0$ and $\delta > 0$. We can take $\delta$ small enough such that 
$H_t (u) \in (-\infty, a(\tau))$ for all $u \in (-\infty, a(\tau) + \delta)$. 
This contradicts the equality $A(H_t(z)) = A(z)\frac{\partial H_t}{\partial z}(z)$. 
\end{proof}

In case (c), if $\lim_{u \nearrow a(\tau)} A(u) = 0$, the question as to whether the relation $E(t) = a(\nu_t)$ holds for all $t > 0$ or not, 
has not been clarified yet. A partial answer is shown in the following proposition.

\begin{prop}\label{partialresult}  We consider the case $(c)$.
Then $a(\nu_t) = a(\tau)$ a.e. with respect to the Lebesgue measure on $[0, \infty)$ and $a(\nu_t) \geq a(\tau)$ for all $t > 0$. 
\end{prop}
\begin{proof}
\textit{Step 1.} First, we prove the following fact: if $ \limsup_{t \rightarrow t_0, t \neq t_0} a(\nu_t) \geq  a(\nu_{t_0})$, then 
$A(z)$ is analytic in the region $(-\infty, a(\nu_{t_0}))$ and moreover, $a(\nu_{t_0}) = a(\tau) ~(=E(t_0)$.  
Fix an arbitrary number $\epsilon \in (0, 1)$. 
Take a sequence $\{t_n \}_{n=1} ^{\infty}$ such that 
$a(\nu_{t_n}) \geq a(\nu_{t_0}) -\frac{\epsilon }{2}$ for all $n \geq 1$ 
and define the sequence of 
analytic functions in $(-\infty, a(\nu_{t_0}) - \epsilon )$ by 
\[ A_n ^\epsilon (z):= \frac{H_{t_n}(z) - H_{t_0}(z)}{t_n - t_0} \]
for $n \geq 1$. 
For any compact set $K \subset \comp \backslash [a(\tau)-\epsilon ,\infty)$, 
we can prove that the sequence $\{ A_n ^\epsilon \}$ is uniformly bounded 
on $K$ for sufficiently large $n$. Hence we obtain the analyticity of $\partial _t H (t_0, z)$ 
in $(-\infty, a(\nu_{t_0}) - \epsilon )$. Since $1 > \epsilon > 0$ is arbitrary, 
we conclude that $\partial _t H (t_0, z)$ is analytic in $(-\infty, a(\nu_{t_0}))$. $A(z)$ has an analytic continuation from $\comp \backslash \real$ 
to $\comp \backslash [a(\nu_{t_0}), \infty)$ by the equality 
$A(z) = \frac{\partial _t H(t_0, z)}{\partial _z H(t_0, z)}$. 
 Now we show $a(\tau) = a(\nu_{t_0})$. 
As explained before, the solution $H_t(z)$ of the ODE exists for all time and for any initial position $z \in \comp \backslash [a(\tau), \infty)$. Therefore, we obtain $a(\nu_t) \geq a(\tau)$ for all $t \in [0, \infty)$. 
Moreover, we can prove that $a(\tau) \geq a(\nu_{t_0})$ by the analyticity of $A(z)$ in $(-\infty, a(\nu_{t_0}))$. 
 
\textit{Step 2.} We note that $a(\nu_t)$ is Borel measurable. 
This is easy since the coefficients of the Taylor expansion of $H_t$ is measurable (by the Cauchy integral formula), 
and $a(\nu_t)$ can be expressed by the limit supremum of them.  
We define a Borel set $B$ by 
\begin{align*}
B &:= \{t \in [0, \infty); \text{ there exist } \epsilon = \epsilon (t) > 0  
\text{ and } \eta = \eta (t) > 0 \text{ such that } \\
   &~~~~~ |a(\nu_t) - a(\nu_s)| > \epsilon \text{ for all } s \text{~satisfying~} 0 < |s-t| < \eta  \}. 
\end{align*}
If $t \in B^c$, $a(\nu_t) = E(t)$ by Step 1. 
It is known that a Borel measurable function on an interval is continuous except for an open set with arbitrary small Lebesgue 
measure by Lusin's theorem (see \cite{Fol}). 
Therefore, the Lebesgue measure of the set $B$ is $0$. $a(\nu_t) \geq a(\tau)$ was already mentioned in the proof of Step 1. 
\end{proof}

So far we have proved that $E(t)=a(\nu_t)$ in generic cases. Next we show an ODE for the function $E(t)$. 
Define by 
\[ 
E_\epsilon (t):= \sup \{u \leq  a(\tau);H_t (u) = a(\tau) -\epsilon  \}
\] 
an approximate family for $\epsilon > 0$. This approximation is needed 
to use the implicit function theorem in the proof of Theorem \ref{E2}.

\begin{lem}\label{convergence2}In case $(a)$ and case $(b)$,  
$E_\epsilon $ and $E$ enjoy the following properties. \\
$(1)$ $E_\epsilon  < E$ for all $\epsilon \in (0, 1)$. 
In addition, $E_\epsilon $ converges to $E$ pointwise as $\epsilon \rightarrow 0$. \\            
$(2)$ $\sup _{\epsilon > 0, t \in I} |E_\epsilon (t)| < \infty$ for any compact set $I \subset [0, \infty)$
\end{lem}
The above lemma is easily proved and we omit its proof. 


\begin{thm}\label{E2} We consider case $(a)$ and case $(b)$. 
Then $E(t)$ satisfies the ODE 
\[
\begin{cases} \frac{d}{dt} E(t) = -A(E(t)) & \text{~for~~} 0 < t < \infty, \\ 
              E(0) = a(\tau). 
\end{cases}
\]
In particular, $E$ is in $C^{\omega }(0, \infty) \cap C[0, \infty)$. 
\end{thm} 
\begin{proof}
We note that the inequality $\frac{\partial H}{\partial u} \geq  1$ holds. 
Then Implicit Function Theorem is applicable to the equation $H = a(\tau) - \epsilon $ 
because $H$ is defined in the open set 
$\{(t,u); 0 < t < \infty, -\infty < u < E(t) \}$ which 
contains $(t, E_\epsilon (t))$ for all $t$.
Therefore, $E_\epsilon $ is in class $C^{\omega}(0, \infty)$ 
and its derivative is 
\begin{equation*}
\frac{dE_\epsilon }{dt}(t) = -\frac{ \partial _t H(t, E_\epsilon (t))}{ \partial _u H(t, E_\epsilon (t))}   = -A(E_\epsilon (t))  
\end{equation*}
by Lemma \ref{equalities}.   
After integrating the above, we take the limit $\epsilon \to 0$ using Lemma \ref{convergence2}, to obtain 
\[ E(t) = \int_t ^{t_1} A(E(s))ds + E(t_1). \] 
This implies that $E$ is in class $C^{\omega}(0, \infty)$ and 
the ODE holds. The right continuity of $E$ at $0$ follows from the fact $\lim_{t \searrow 0} H_t (z) = z$. 
\end{proof}


\section{Time-dependent and time-independent properties of monotone convolution semigroup} \label{time}
In classical probability theory, it is often true that a property of a convolution semigroup 
$\mu_t$ is completely determined at an instant. Such a property is called a time-independent property.  
In this section, we prove such properties for monotone convolution semigroups. 
\begin{lem}\label{positive support}
Let $\{\mu_t \}_{t \geq 0}$ be a weakly continuous $\rhd$-convolution semigroup with $\mu_0 = \delta_0$, and 
$A(z)$ be the associated vector field. If there exists $t_0 > 0$ such that $\supp \mu_{t_0} \subset [0, \infty)$, 
then $A(z)$ is analytic in $\comp \backslash [0, \infty)$.  
\end{lem}
\begin{proof} 
We have $\supp \mu_{\frac{t_0}{n}} \subset [0, \infty)$ by Corollary \ref{positivity} (1). Let $A_n(z)$ be defined by 
$A_n(z) := (H_{\frac{t_0}{n}}(z) - z) / \frac{t_0}{n}$. 
$A_n$ is analytic in $\comp \backslash [0, \infty)$. 
By definition $A(z) = \lim_{n \to \infty} A_n(z)$ for $z \in \comp \backslash \real$.  
By Montel's theorem, it suffices to show that the RHS is uniformly bounded on each compact subset of $\comp \backslash [0, \infty)$. 
Fix an arbitrary compact set $K \subset \comp \backslash [0, \infty)$. 
By Lemma \ref{support}, $\supp \eta_{\frac{t_0}{n}} \subset [0, \infty)$. 
Since $H_t (i) = b_t + i(1 + \eta_t (\real))$ is differentiable, 
there exist $M, M' > 0$ such that $\frac{\eta_t (\real)}{t} \leq M$ and 
$\Big{|}\frac{b_t}{t} \Big{|} \leq M'$ for all $t \in [0, t_0]$. Then 
\begin{align*}
|A_n(z)| &\leq \Big{|}\frac{n}{t_0}b_{\frac{t_0}{n}} \Big{|} +  \Bigg{|} \int_0 ^{\infty} \frac{1 + xz}{x-z} \frac{n}{t_0} \eta_{\frac{t_0}{n}}(x)\Bigg{|} \\    
         &\leq M' + L'
\end{align*}
for all $n$ and $z \in K$. $L' > 0$ is a constant dependent only on $K$. 
\end{proof}

Using Proposition \ref{positive support2} and Lemma \ref{positive support}, one can prove the monotone analogue of 
subordinator theorem. For the classical version, the reader is referred to Theorem 24.11 of \cite{Sat}. 
\begin{thm}\label{subordinator}.  
Let $\{\mu_t \}_{t \geq 0}$ be a weakly continuous $\rhd$-convolution semigroup 
with $\mu_0 = \delta_0$. Then the following statements are equivalent: 
\begin{itemize}
\item[(1)] there exists $t_0 > 0$ such that $\supp\mu_{t_0} \subset [0, \infty)$;
\item[(2)] $\supp \mu_t \subset [0, \infty)$ for all $0 \leq t < \infty$;
\item[(3)] $\supp \tau \subset [0, \infty)$,  $\tau(\{0 \}) = 0$, $\int_0 ^{\infty}\frac{1}{x}d\tau (x) < \infty$ and 
$\gamma \geq \int_0 ^{\infty}\frac{1}{x}d\tau (x)$. 
\end{itemize}
\end{thm}
\begin{rem}
(i) The equality $\tau(\{0 \}) = 0$ in condition (3) means that there is no component of a Brownian motion in the L\'{e}vy-Khintchine formula. \\
(i\hspace{-.1em}i) The equivalence also holds in the classical and Boolean L\'{e}vy-Khintchine formulae. 
In the free case, however, (1) and (2) are not equivalent (see Section \ref{time2}). 
\end{rem}
\begin{proof} 
We note that (3) is equivalent to 
(3'): $A$ is analytic in $\comp \backslash [0, \infty)$ and $A < 0$ on $(-\infty, 0)$, by an argument in Proposition \ref{positive support2}. 

$(1) \Rightarrow (2), ~(3')$: If $\{\mu_t \}$ is a delta measure, then the statement follows immediately. 
We assume that $\mu_t$ is not a delta measure for some $t>0$. 
This is equivalent to assuming that $\mu_t$ is not a delta measure for all $t>0$.  
Then $\tau$ is a nonzero positive finite measure. 
$A(z)$ is analytic in $\comp \backslash [0, \infty)$ by Lemma \ref{positive support}, and hence, $\supp \tau \in [0, \infty)$.  
There are three possible cases: (a) $A(u) > 0$ on $(-\infty, 0)$;  (b) $A(u_0)$ = 0 for some $u_0 \in (-\infty, 0)$ and $A(u) < 0$ on ($-\infty, u_0)$ and 
  $A(u) > 0$ on $(u_0, 0)$;   (c) $A(u) < 0$ on $(-\infty, 0)$.  

In case $(a)$ and case $(b)$, we have $a(\mu_t) < 0$ for all $t > 0$ by Theorem \ref{Diffeq} (2). 
In case $(c)$, we have $a(\mu_t) \geq a(\tau) \geq 0$ again by Theorem \ref{Diffeq} (2). 
Hence only case (c) has no contradiction to the assumption. 

$(3') \Rightarrow (1)$: This proof was actually done in the end of the proof of $(1) \Rightarrow (2)$. 
 \end{proof}
   
We can prove that the lower boundedness of the support is determined at one instant.    
\begin{thm}\label{unboundedness below}  
Let $\{\mu_t \}_{t \geq 0}$ be a weakly continuous $\rhd$-convolution semigroup 
with $\mu_0 = \delta_0$. Then the following statements are equivalent: 
\begin{itemize}
\item[$(1)$] there exists $t_0 > 0$ such that $\supp\mu_{t_0}$ is bounded below; 
\item[$(2)$] $\supp \mu_t$ is bounded below for all $0 \leq t < \infty$;  
\item[$(3)$] $\supp \tau$ is bounded below.  
\end{itemize}
\end{thm}
\begin{rem}
The same kind of theorem also holds in the free and Boolean cases. The classical case is exceptional since the condition (3) needs to be replaced by $\supp \tau \subset [0, \infty)$, $\tau(\{0 \}) = 0$ and $\int_{-1}^1 \frac{1}{|x|}d\tau(x) < \infty$ \cite{Sat}. 
Therefore, the boundedness below is not mapped bijectively by the monotone analogue of Bercovici-Pata bijection 
defined in Section \ref{connec}. 
\end{rem}
\begin{proof}
$(1) \Rightarrow (3)$: When $a(\mu_{t_0}) \geq 0$, the claim follows from Theorem \ref{subordinator}. 
We consider the case $a(\mu_{t_0}) < 0$. 
By Proposition \ref{estimate}, we have 
$a(\mu_t) \geq a(\mu_{t_0}) > -\infty$ for all $t \leq t_0$. 
By the same argument as in Lemma \ref{positive support}, one can show that $A$ is analytic in $(-\infty, a(\mu_{t_0}))$. \\
$(3) \Rightarrow (2)$: The lower boundedness of the support of 
$\mu_t$ for all $t \geq 0$ comes from Theorem \ref{Diffeq}. 
\end{proof}

Next we consider the symmetry around the origin. We say that a measure $\mu$ on the real line is symmetric 
if $\mu(dx) = \mu(-dx)$. The proof depends on the assumption of compact support. 
We could not prove the result for all probability measures.

\begin{thm}\label{symm} 
Let $\{\mu_t \}_{t \geq 0}$ be a weakly continuous $\rhd$-convolution semigroup 
with $\mu_0 = \delta_0$. 
We assume that the support of each $\mu_t$ is compact $($this is a time-independent property$)$. 
Then the following statements are all equivalent. 
\begin{itemize}
\item[$(1)$] There exists $t_0 > 0$ such that $\mu_{t_0}$ is symmetric.
\item[$(2)$] $\mu_{t}$ is symmetric for all $t > 0$.
\item[$(3)$] $\gamma = 0$ and $\tau$ is symmetric.
\end{itemize}
\end{thm}
\begin{proof}
We prove this theorem in terms of moments. We use the representation of the vector field $A(z) = -\gamma + \int \frac{1}{x-z}d\sigma (x)$, $d\sigma(x) = (1+x^2)d\tau (x)$, 
where $\sigma $ has a compact support. 
We use the notation $m_n(t) = m_n(\mu_t)$ for simplicity. 
We notice that the symmetry is equivalent to the vanishment of odd moments for a compactly supported measure. 
Define a sequence $\{r_n \}_{n=1}^{\infty}$ by $r_1  := \gamma$, $r_n  := m_{n-2}(\sigma)$ for $n \geq 2$.  
Then $A(z)= -\sum_{n = 1} ^{\infty} \frac{r_n}{z^{n-1}}$. By Lemma \ref{equalities} (2), we get differential equations $\frac{dm_0(t)}{dt} = 0$ and 
\begin{equation}\label{system of ODEs}
\frac{dm_n(t)}{dt} = \sum_{k=1} ^n k r_{n-k + 1} m_{k-1}(t) \text{~~for $n \geq 1$}
\end{equation}
with initial conditions $m_0(0) = 1$ and $m_n (0) = 0$ for $n \geq 1$. 

Now we prove the implications $(1) \Rightarrow (2)$ and $(1) \Rightarrow (3)$. 
We can easily prove that $m_{2n + 1}(t_0) = 0$ and $r_{2n + 1} = 0$ for $n \geq 0$, and then $m_{2n + 1}(t) = 0$ for all $t > 0$ and $n \geq 0$. Then $\sigma$ and $\mu_t$ are both symmetric for all $t > 0$. 
The proof of the implication $(3) \Rightarrow (2)$ runs by a similar argument. 
\end{proof}

We show some time-dependent properties. 
\begin{prop}
$(1)$ Absolute continuity is a time-dependent property. \\ 
$(2)$ Existence of an atom is a time-dependent property. 
\end{prop}
\begin{proof} 
There is an example \cite{Has2}. Let $\{\mu_t \}_{t \geq 0}$ be the monotone convolution semigroup defined by 
\begin{equation}
H_t ^{(\alpha,1,c)}(z)= c + \{(z - c)^\alpha + t \}^{\frac{1}{\alpha}} \text{~ for~} 0 < \alpha < 1.
\end{equation}
Then $\mu_t$ contains an atom for $0 \leq t < |c|^\alpha$ and $\mu_t$ is absolutely continuous for $t \geq |c|^\alpha$.  
\end{proof}

The property $m_{2n}(\mu) = \int_{\real}x^{2n}\mu(dx) < \infty$ is also time-independent. That is, we prove the following theorem 
which is also true in classical and free probabilities \cite{Ben,Sha}. In addition, this also extends Theorem 4.9 in \cite{Mur3} to higher order moments. 
\begin{thm}\label{moment19}
Let $\{\mu_t \}_{t \geq 0}$ be a weakly continuous $\rhd$-convolution semigroup 
with $\mu_0 = \delta_0$ and let $n \geq 1$ be a natural number. Then the following statements are equivalent: 
\begin{itemize}
\item[$(1)$] there exists $t_0 > 0$ such that $m_{2n}(t_0) < \infty$; 
\item[$(2)$] $m_{2n}(t) < \infty $ for all $0 < t < \infty$;  
\item[$(3)$] $m_{2n}(\tau) < \infty$.  
\end{itemize}
\end{thm}
\begin{proof} 
$(1) \Rightarrow (2)$: We use the notation $\mu_{t}^y := \delta_y \rhd \mu_t$ introduced in (\ref{eq0111}).  For $0 \leq t \leq t_0$, we set $\lambda = \mu_{t_0 - t}$ and $\nu = \mu_t$. 
Then we obtain $ \int \int x^{2n}\mu_{t}^y(dx)\mu_{t_0 - t}(dy) = \int_{\real}x^{2n}\mu_{t_0} (dx) < \infty$, 
which implies $m_{2n}(\mu_t ^y) < \infty$ for some $y \in \real$. By Proposition \ref{asymptotic2}, we obtain 
$m_{2n}(t) < \infty$ for $0 \leq t \leq t_0$. For arbitrary $0 < s < \infty$, we can write $s = kt_0 + t$ with $k \in \nat$ and $0 \leq  t < t_0$. Then we have $m_{2n}(s) < \infty$ by Proposition \ref{moment21}. \\
$(2) \Rightarrow (3)$: We first note that $m_{k}(t)$ is a Borel measurable function of $t \leq 0$ since $\mu_t$ is weakly continuous. 
Moreover, we show that there exist $r_1, \cdots, r_{2n} \in \real$ such that 
\begin{equation}\label{moment23}
m_{l}(t) = \sum_{k =1} ^{l} \sum_{1 = i_0 < i_1 < \cdots < i_{k-1} < i_k = l +1} 
\frac{t^k}{k!}\prod_{p = 1} ^{k} i_{p-1} r_{i_p - i_{p-1}}
\end{equation}
for $1 \leq l \leq 2n$. 
For the proof we use the equality 
\begin{equation}\label{moment22}
m_l(t + s) = m_l(t) + m_l(s) + \sum_{k = 1} ^{l-1} \sum_{\substack{j_0 + j_1 + \cdots +j_k = l - k, \\  0 \leq j_p, ~0 \leq p \leq k }} m_k(t) m_{j_0}(s)\cdots m_{j_k}(s)  
\end{equation}
for $1 \leq l \leq 2n$. For $l = 1$, (\ref{moment22}) becomes $m_1(t + s) = m_1(t) + m_1(s)$. This is Cauchy's functional equation and there exists $r_1 \in \real$ such that $m_1(t) = r_1 t$ by the measurability (for a simple proof of Cauchy's functional equation, see \cite{A-O}).  
We assume that there exist $r_1, \cdots, r_q \in \real$ such that (\ref{moment23}) holds for $1 \leq l \leq q$. 
For an arbitrary $r'_{q + 1} \in \real$, we define 
\begin{equation}
\widetilde{m}_{q+1}(t) := r'_{q + 1}t + \sum_{k = 2} ^{q+1} \sum_{1 = i_0 < i_1 < \cdots < i_{k-1} < i_k = q + 2} 
\frac{t^k}{k!}\prod_{p = 1} ^{k} i_{p-1} r_{i_p - i_{p-1}}. 
\end{equation}
Then the equality 
\begin{equation}\label{moment221}
\widetilde{m}_{q+1}(t + s) = \widetilde{m}_{q+1}(t) + \widetilde{m}_{q+1}(s) + \sum_{k = 1} ^{q} \sum_{\substack{j_0 + j_1 + \cdots +j_k = q + 1 - k, \\  0 \leq j_l, ~0 \leq l \leq k }} m_k(t) m_{j_0}(s)\cdots m_{j_k}(s)  
\end{equation}
holds; this will be proved soon later in Proposition \ref{mmom}. Therefore, (\ref{moment22}) and (\ref{moment221}) imply that 
$m_{q + 1}(t + s) - \widetilde{m}_{q+1}(t + s) = m_{q + 1}(t) - \widetilde{m}_{q+1}(t) + m_{q +  1}(s) - \widetilde{m}_{q+1}(s)$. This is again Cauchy's functional equation, and hence, there exists $r''_{q + 1} \in \real$ such that $m_{q + 1}(t) = \widetilde{m}_{q+1}(t) + r''_{q + 1} t$. The above argument runs until $q = 2n-1$, and then we conclude that there exist $r_1, \cdots, r_{2n} \in \real$ such that (\ref{moment23}) holds for $1 \leq l \leq 2n$. 

By the equality $\frac{\partial G}{\partial t}(t, z) = A(z)\frac{\partial G}{\partial z}(t,z)$ we obtain 
$A(z) = \frac{G(1, z) - \frac{1}{z}}{\int_0 ^1 \frac{\partial G}{\partial z}(s,z)ds}$, which implies 
\begin{equation}
A(z) = -\frac{\frac{m_1(1)}{z^2} + \cdots + \frac{m_{2n}(1)}{z^{2n + 1}} + o(|z|^{-(2n+1)})}{\frac{1}{z^2} + \frac{2\int_0 ^1 m_1(s)ds}{z^3} + \cdots + \frac{(2n+1)\int_0 ^1 m_{2n}(s)ds}{z^{2n+2}} + \int_0 ^1 R_s(z)ds },  
\end{equation} 
where $R_s(z)$ is defined by $R_s(z) = \frac{2n+1}{z^{2n+2}} \int_{\real} \frac{x^{2n + 1}}{z - x} \mu_s(dx) + \frac{1}{z^{2n+1}} \int_{\real} \frac{x^{2n + 1}}{(z - x)^2} \mu_s(dx)$. We prove a property of $R_s(z)$ here. Since $m_{2n}(s)$ is a polynomial, $x^{2n}$ is integrable with respect to the measure $\mu_s(dx)ds$ on $\real \times [0, t]$. 
Easily we can show that $\int_0 ^1 R_s(iy)ds = o(y^{-(2n+2)})$ by the dominated convergence theorem. Therefore, there exist $u_1, \cdots, u_{2n} \in \real$ such that $A(iy) = u_1 + \frac{u_2}{iy} + \cdots + \frac{u_{2n}}{(iy)^{2n-1}} + o(y^{-(2n-1)})$. 
By Proposition \ref{asymptotic2}, we have $m_{2n}(\tau) < \infty$ (the equivalence between (2) and (3) in Proposition \ref{asymptotic2} is true for $A(z)$. The proof needs no changes). \\
$(3) \Rightarrow (2)$: Since $m_{2n}(\tau) < \infty$, we have the expansion $A(z) = u_1 + \frac{u_2}{z} + \cdots + \frac{u_{2n}}{z^{2n-1}} + Q(z)$, where $Q(z) := \frac{1}{z^{2n-1}}\int_{\real} \frac{x^{2n-1}}{x - z}(1 + x^2)\tau(dx)$.  
We obtain  
\begin{equation} \label{exp12}
  H_t(z) = z + u_1 t + \int_0 ^t \frac{u_2}{H_s(z)}ds + \cdots + \int_0 ^t \frac{u_{2n}}{H_s(z)^{2n - 1}}ds + \int_0 ^t Q(H_s(z))ds 
\end{equation}
from the equality $\frac{d}{dt}H_t(z) = A(H_t(z))$. 
We can prove that $\sum_{k = p} ^{2n-1} \int_0 ^t \frac{u_{k+1}}{H_s(iy)^k}ds + \int_0 ^t Q(H_s(iy))ds = o(y^{-(p-1)})$ since $|\int_0 ^t \frac{1}{H_s(iy)^k}ds| \leq \frac{t}{y^{k}}$. In addition,  $\int_0 ^t Q(H_s(iy))ds = o(y^{-(2n-1)})$ for any $t > 0$ by the dominated convergence theorem.  
Now we show by induction that there exist polynomials $c_k(t)$ of $t$ $(1 \leq k \leq 2n)$ such that 
\begin{equation}\label{exp121}
H_t(z) = z + c_1(t) + \frac{c_2(t)}{z} + \cdots + \frac{c_{2n}(t)}{z^{2n-1}} +  o(|z|^{-(2n-1)})~~~~(z = iy,~y \to \infty) 
\end{equation}
for any $t > 0$.  
First $H_t(iy) = iy + u_1 t + \frac{u_2 t}{iy} + o(\frac{1}{y})$ holds by (\ref{exp12}). Next we assume that there exist polynomials $c_k(t)$ of $t$ $(1 \leq k \leq 2q)$ such that 
\begin{equation}\label{exp13}
H_t(z) = z + c_1(t) + \frac{c_2(t)}{z} + \cdots + \frac{c_{2q}(t)}{z^{2q-1}} + P_t(z), 
\end{equation}
where $P_t(iy) = o(y^{-(2q-1)})$ for any $t > 0$. 
We can write $P_t(z) = \frac{1}{z^{2q-1}} \int_{\real}\frac{x^{2q-1}}{x - z}\rho_t(dx)$, where $\rho_t$ is the positive finite measure in Proposition \ref{asymptotic2} (2). Then we obtain the asymptotic behavior $\int_0 ^t P_s(iy)ds = o(y^{-(2q - 1)})$. Substituting (\ref{exp13}) into the right hand side of (\ref{exp12}), we obtain the expansion 
\begin{equation}\label{exp14} 
H_t(z) = z + b_1(t) + \frac{b_2(t)}{z} + \cdots + \frac{b_{2q + 2}(t)}{z^{2q + 1}} + o(|z|^{-(2q + 1)}),  
\end{equation}  
where $b_k(t)$ is a polynomial of $t$ (we note that $b_k(t) = c_k(t)$ holds for $1 \leq k \leq 2q$ by the uniqueness of the expansion).  This induction goes until $q = n-1$ and we obtain (\ref{exp121}). The conclusion follows from Proposition \ref{asymptotic2}. 
\end{proof}
\begin{rem}
We have proved that $m_k(t)$ is a polynomial of $t$ in the proof of $(2) \Rightarrow (3)$. This property might seem to be too strong: what we needed was the integrability of $m_k(t)$ in a finite interval. The author however could not find an alternative proof of the integrability.  
\end{rem}

The following result completes the above theorem.  
\begin{prop}\label{mmom}
For any complex numbers $r_n$, $n\geq 1$, $m_n(t)$ defined by 
\begin{equation}\label{mmom1}
m_{n}(t) = \sum_{k =1} ^{n} \sum_{1 = i_0 < i_1 < \cdots < i_{k-1} < i_k = n +1} 
\frac{t^k}{k!}\prod_{p = 1} ^{k} i_{p-1} r_{i_p - i_{p-1}}
\end{equation}
satisfy the equality 
\begin{equation}\label{mmom2}
m_n(t + s) = m_n(t) + m_n(s) + \sum_{k = 1} ^{n-1} \sum_{\substack{j_0 + j_1 + \cdots +j_k = n - k, \\  0 \leq j_p, ~0 \leq p \leq k }} m_k(t) m_{j_0}(s)\cdots m_{j_k}(s)  
\end{equation} 
for any $n \geq 1$. 
\end{prop}
\begin{proof}
Every series in this proof is a formal power series. 
We define $A(z)=-\sum_{z=1}^\infty \frac{r_n}{z^{n-1}}$. We solve the differential equation (\ref{ODE}) in the sense of formal power series. Then 
the solution $H_t(z)$ of the form $H_t(z) = \sum_{n=-1} ^\infty \frac{a_n(t)}{z^n}$ uniquely exists. It is easy to prove that $H_{t+s}(z) = H_t(H_s(z))$ in the sense of formal power series with respect to $t,s,z$. If we define $G_t(z)$ by $\frac{1}{H_t(z)}$, then Lemma \ref{equalities} holds by the same proof. 
We can easily prove that $m_n(t)$ are given by $G_t(z) = \sum_{n=0}^\infty\frac{m_n(t)}{z^{n+1}}$ using the equality (2) in Lemma \ref{equalities}. 
(\ref{mmom2}) follows from the power series expansion of $G_{t+s}(z) = G_t(\frac{1}{G_s(z)})$.    
\end{proof}

\section{Strictly stable distributions}\label{stable}
Let $b \in \comp$, $c \in \comp$ and $\alpha \in \real$ be constants such that $\alpha \neq 0$. 
We consider a $\rhd$-infinitely divisible probability distribution $\mu ^{(\alpha,b,c)}$
by the associated vector field 
\begin{equation}\label{computa0}
A^{(\alpha,b,c)}(z):= \frac{b}{\alpha} (z-c)^{1-\alpha}, 
\end{equation}
where  $z^s$ is defined by $z^s = \exp(s \log z)$ for 
$z \in \comp \backslash \{x \in \real; x \geq 0 \}$.
The range of the angle of $z$ is chosen to be $0 < \arg z < 2\pi$. 
(Of course the factor $\frac{b}{\alpha}$ can be replaced by merely $b$; however, we use this notation since (\ref{computa1}) becomes rather simple.) 
In order that $A^{(\alpha,b,c)}$ becomes the associated vector field to a $\rhd$-infinitely divisible distribution, the following conditions 
are necessary and sufficient:  
\begin{itemize}
\item[(a)] $A^{(\alpha, b, c)}$ maps $\com+$ into $\com+ \cup \real;$ 
\item[(b)] $\lim_{y \rightarrow \infty} \frac{ \im A^{(\alpha, b, c)}(x + iy)}{y} = 0$ for some $x$.
\end{itemize} 
By careful observation upon the motion of angles, we can see that $A^{(\alpha, b, c)}$ satisfies (a) and (b) if and only if
\begin{itemize} 
\item[$(1)$] $\im c \leq 0$, 
\item[$(2)$] $0 < \alpha \leq 2$, 
\item[$(3)$] $0 \leq \arg b \leq \alpha \pi$ for $0 < \alpha \leq 1$ and  $(\alpha-1)\pi \leq \arg b \leq \pi$ for $1 < \alpha \leq 2$,  
\end{itemize}
except for the case $\alpha = 1$. If $\alpha = 1$, $A^{(1, b, c)}$ 
does not depend on $c$ and the condition (1) is not needed.  
We can write explicitly the corresponding reciprocal Cauchy transform: 
\begin{equation}\label{computa1}
H ^{(\alpha,b,c)}(z)= c + \{(z - c)^\alpha + b \}^{\frac{1}{\alpha}}. 
\end{equation}
When we consider the corresponding convolution semigroup $\{ \mu_t ^{(\alpha,b,c)} \}_{t \geq 0}$, 
the formula (\ref{computa1}) becomes 
\begin{equation}\label{computa2}
H_t ^{(\alpha,b,c)}(z)= c + \{(z - c)^\alpha + bt \}^{\frac{1}{\alpha}}. 
\end{equation}
This family is an extension of deformed arcsine laws in \cite{Mur3} ($\im c = 0$, $\alpha = 2$), 
Cauchy distributions ($\alpha = 1$, $b = \beta i$ with $\beta > 0$) 
and delta measures ($\alpha = 1$, $\im b = 0$). Moreover, this family gives good examples when we study support properties of general $\rhd$-infinitely divisible distributions \cite{Has1}.

We show that the family  $\{\mu ^{(\alpha, b, 0)} \}$ gives all strictly monotone stable distributions which we define now. 
Let $D_\lambda$ be the dilation operator defined by 
\begin{equation}
D_\lambda \mu (B) = \mu (\lambda^{-1}B), 
\end{equation}
where $B$ is an arbitrary Borel set and $\mu$ is an arbitrary Borel measure. 
\begin{defi}
Let $\mu$ be a $\rhd$-infinitely divisible distribution. Then 
there exists a unique weakly continuous $\rhd$-convolution semigroup $\{\mu_t \}_{t \geq 0}$ such that $\mu_1 = \mu$ and $\mu_0 = \delta_0$. 
$\mu$ is called a strictly $\rhd$-stable distribution 
if for any $a>0$ there exists $b(a) > 0$ such that 
\begin{equation}\label{stabledef}
\mu_{a} = D_{b(a)}\mu. 
\end{equation}
(\ref{stabledef}) is equivalent to the following equality: 
\begin{equation}
H_{\mu_{a}} (z) = b(a) H_\mu (b(a)^{-1}z) \text{~ for all~} z \in \com+. 
\end{equation}
We often write $H_t = H_{\mu_{t}}$ for simplicity.
\end{defi}

\begin{rem}
We do not treat unbounded operators which will be interesting in the study of strictly $\rhd$-stable distributions; 
we deal with only probability distributions. 
\end{rem}

\begin{lem}\label{stable1} Assume that $\mu$ is a strictly $\rhd$-stable distribution with $\mu \neq \delta_0$. \\
$(1)$ $b(a)$ is unique for each $a > 0$ and $b(a)$ is a continuous function of $a$.  \\
$(2)$ It holds that $H_{at}(z) = b(a) H_t(b(a)^{-1}z)$ for all $a > 0$, $t \geq 0$ and $z \in \com+$. \\ 
$(3)$ There exists some $h \in \real$ such that $b(a)= a^h$ for all $a > 0$, 
$t\geq 0$. 
\end{lem}
\begin{proof}
(1) The proof of the uniqueness of $b(a)$ given below is almost the same as in Lemma 13.7 in \cite{Sat}. 
Fix $a > 0$. Assume that there exist $b > b'$ such that
$H_a(z) = b H_1 (b^{-1}z) = b' H_1 (b'^{-1}z) $ for all $z \in \com+$. 
Then we have $\mu(\frac{dx}{b}) = \mu(\frac{dx}{b'})$, and hence, we 
get $\mu(c^n dx) = \mu(dx)$ with $c= \frac{b}{b'} > 1$ for all $n \in \nat$. 
Letting $n \rightarrow \infty$, we have $\mu = \delta_0$, which is a contradiction.  
Then we have the uniqueness of $b$. The proof of the continuity of $b(a)$ is the same as Lemma 13.9 in \cite{Sat} and we omit the proof.  \\
$(2)$ We fix $a > 0$. 
Define two families of probability measures $\nu_t:=\mu_{at}(dx)$ 
and $\lambda_t:= \mu_t \Big{(}\frac{dx}{b(a)} \Big{)}$. 
Since $\mu$ is strictly $\rhd$-stable, we have $\nu_1 = \lambda_1$. 
Moreover, both $\{ \nu_t \}$ and $\{ \lambda_t \}$ constitute monotone convolution semigroups. 
Therefore, we obtain $\nu_t = \lambda_t$ for all $t \geq 0$ by the uniqueness result obtained in \cite{Bel}. \\
$(3)$ 
By the result (2) it holds that 
\begin{equation}
H_{aa't}(z) = b(a)H_{a't}(b(a)^{-1}z) = b(a)b(a')H_t(b(a)^{-1}b(a')^{-1}z)
\end{equation} 
for all $a,a' > 0$ and $t \geq 0$. Therefore, by (1) we have  
\begin{equation}\label{Cauchy}
b(aa') = b(a)b(a') 
\end{equation}
for all $a, a' > 0$.  It is a well known fact that a continuous function satisfying the equation (\ref{Cauchy}) 
is of the form $b(a) = a^h$. 
\end{proof}
\begin{defi}
The reciprocal of $h$ in Lemma \ref{stable1} is called the index of $\mu$. 
We denote the index by $\alpha$ and in this case we call $\mu$ a strictly $\rhd$-$\alpha$-stable distribution.
\end{defi}

Assume that $\mu$ is a $\rhd$-infinitely divisible distribution.   
Let $A$ be the associated vector field in (\ref{inf.div}).
The following equivalent conditions are useful in the classification of strictly $\rhd$-stable distributions: 
\begin{itemize}
\item[(1)] $\mu$ is a strictly $\rhd$-$\alpha$-stable distribution; 
\item[(2)] $H_{at} (z) = a^{\frac{1}{\alpha}}H_t (a^{- \frac{1}{\alpha}}z)$ for all $z \in \com+$; 
\item[(3)] $A(z) = a^{\frac{1}{\alpha} -1 } A(a^{- \frac{1}{\alpha}}z)$ for all $z \in \com+$. 
\end{itemize}
\begin{thm}\label{stable2}
Assume that $\mu$ is a strictly $\rhd$-stable distribution with $\mu \neq \delta_0$. 
Then the index $\alpha$ of $\mu$ satisfies $0 < \alpha \leq 2 $. 
Moreover, there exists $b \in \comp$ such that $\mu = \mu ^{(\alpha,b,0)}$,  
where $b$ satisfies the following conditions:
\begin{itemize}
\item[$\cdot$] $ 0 \leq  \arg b \leq \alpha\pi $ if $0 < \alpha \leq  1$, 
\item[$\cdot$] $(\alpha-1)\pi \leq \arg b \leq \pi$ if $1 < \alpha \leq 2$.
\end{itemize}
\end{thm}  
\begin{proof} 
We have the equation $A(z) =  a^{\frac{1}{\alpha} -1 } A(a^{- \frac{1}{\alpha}}z)$ for all $a > 0$ 
and $z \in \com+$ since 
$\mu$ is a strictly $\rhd$-$\alpha$-stable distribution. Differentiating the equation w.r.t. $a$, we obtain 
$A'(a^{- \frac{1}{\alpha}}z) = \frac{1-\alpha}{a^{- \frac{1}{\alpha}}z}A(a^{- \frac{1}{\alpha}}z)$. 
Putting $a=1$,  we obtain 
\begin{equation}\label{diff}
A'(z) = \frac{1-\alpha}{z}A(z).
\end{equation} 
It follows from the differential equation (\ref{diff}) that 
\begin{equation}
A(z)= \frac{b}{\alpha}z^{1-\alpha}
\end{equation}
for some constant $b \in \comp$. As explained in the beginning of this section, we obtain the conclusion. 
\end{proof}
\begin{rem}
(1) The characterization of strictly $\rhd$-stable distributions in terms of the associated vector field $A(z)= b z^{1-\alpha}$ is very similar to the cases of free \cite{Be-Vo} and boolean \cite{S-W}. \\
(2) Cauchy distributions are strictly $\rhd$-1-stable distributions; this is also the case in classical, free and boolean cases.  
\end{rem}

\section{Connection to infinite divisibility in classical probability theory }\label{connec}
Now we consider the correspondence between commutative probability theory and 
monotone probability theory. The usual L\'{e}vy-Khintchine formula is given by
\begin{equation}\label{khintchine}
\widehat{\mu}(u) = \exp\Big{(}i\gamma u + \int_{\real}\big{(}e^{ixu} - 1 - \frac{ixu}{1 + x^2} \big{)}\frac{1 + x^2}{x^2}\tau(dx) \Big{)},
\end{equation}
where $\gamma \in \real$ and  $\tau$ is a positive finite measure. 
We note that the L\'{e}vy-Khintchine formula in monotone probability theory is given by 
\begin{equation}\label{mkhintchine}
A(z) = -\gamma + \int_{\real}\frac{1 + xz}{x-z}\tau(dx),
\end{equation}
where $(\gamma, \tau)$ satisfies the same conditions in (\ref{khintchine}).
The correspondence between the commutative case and the monotone case becomes clearer if we use the notation of (\ref{khintchine}). 
For instance, the condition of $(\gamma, \tau)$ for the positivity of an infinitely divisible distribution can be written as (see Theorem 24.11 in \cite{Sat})
\begin{equation}
\begin{split}
&\supp \tau \subset [0, \infty),~ \int_0 ^1 \frac{1}{x}\tau(dx) < \infty, \\
&\tau(\{0 \}) = 0, ~\gamma \geq \int_0 ^{\infty} \frac{1}{x}\tau(dx). 
\end{split}
\end{equation}
These conditions are completely the same as in Theorem \ref{subordinator}.  
Then it is natural to define the monotone analogue of the Bercovici-Pata bijection 
(for the details of the Bercovici-Pata bijection, the reader is referred to \cite{Be-Pa}.) 
Let $ID(\rhd)$ be the set of all monotone infinitely divisible distributions; 
let $ID(\ast)$ be the set of all infinitely divisible distributions.
We define a map $\Lambda_M : ID(\ast) \to ID(\rhd)$ by sending the pair 
$(\gamma, \tau)$ in (\ref{khintchine}) to the pair $(\gamma, \tau)$ in (\ref{mkhintchine}) similarly to the Bercovici-Pata bijection. 
This map enjoys nice properties. 
\begin{thm}\label{Lambda} 
$\Lambda_M$ satisfies following properties.
\begin{itemize} 
\item[$(1)$] $\Lambda_M$ is continuous;  
\item[$(2)$] $\Lambda_M(\delta_a) = \delta_a$ for all $a \in \real$; 
\item[$(3)$] $D_\lambda \circ \Lambda_M = \Lambda_M \circ D_\lambda$ for all $\lambda > 0$.
\item[$(4)$] $\Lambda_M$ maps the Gaussian with mean 0 and variance $\sigma^2$ to the arcsine law with mean 0 and variance $\sigma^2$; 
\item[$(5)$] $\Lambda_M$ maps the Poisson distribution with parameter $\lambda$ to the monotone 
Poisson distribution with parameter $\lambda$;  
\item[$(6)$] $\Lambda_M$ gives a one-to-one correspondence between the set 
$\{\mu \in ID(\ast); \supp \mu \subset [0, \infty) \}$ and the set 
$\{\nu \in ID(\rhd); \supp \nu \subset [0, \infty) \}$. 
\item[$(7)$] For all $\alpha \in (0, 2)$, $\Lambda_M$ gives a one-to-one correspondence between strictly $\alpha$-stable distributions and monotone strictly $\alpha$-stable distributions. 
\item[$(8)$] If $\supp \tau $ is compact, the symmetry of $\mu \in ID(\ast)$ is equivalent to the symmetry of $\Lambda_M(\mu)$. 
\item[$(9)$] For each $n \geq 1$, $\Lambda_M$ gives a one-to-one correspondence between the set $\{\mu \in ID(\ast); \int_{\real}x^{2n}\mu(dx) < \infty \}$ and the set $\{\nu \in ID(\rhd); \int_{\real}x^{2n} \nu(dx) < \infty \}$. 
\end{itemize} 
\end{thm}
\begin{rem} Since monotone convolution is non-commutative, $\Lambda_M$ does not preserve 
the structure of convolution: $\Lambda_M(\mu \ast \lambda) \neq \Lambda_M(\mu) \rhd \Lambda_M (\lambda)$ for some $\mu,~\lambda$. \\
\end{rem}
\begin{proof}
(1) It is known that the convergence of a sequence $\{\mu_n \} \subset ID(\ast)$ to some $\mu$ implies the 
convergence of the corresponding pair $(\gamma_n,\tau_n)$ to some $(\gamma, \tau)$. Now we have the family of ODEs driven by 
\[
A_n(z) = -\gamma_n + \int_{\real} \frac{1 + xz}{x -z}d\tau_n(x);
\] 
we denote the flow by $\{H_{n,t} \}$. Since $(\gamma_n, \tau_n)$ converges to $(\gamma, \tau)$, $A_n$ converges locally uniformly to $A$. 
By the basic result of the theory of ODE, it holds that $H_{n,1}(z) \to H_1(z)$ locally uniformly, which implies that
$\mu_n$ converges weakly to $\mu$. \\
(2), (4), (5) The proofs are easy. \\
(3) For a weakly continuous monotone convolution semigroup $\{\mu_t \}$ with $\mu_1 = \mu$ and $\mu_0 = \delta_0$, we have 
$H_{D_{\lambda} \mu_t} (z) = \lambda H_{\mu_t} (\lambda ^{-1}z)$. Then the associated vector field is transformed to 
$A_\lambda (z) =\lambda A(\lambda ^{-1}z)$. We use the notation  
$A(z)= -\gamma + \int \Big{(}\frac{1}{x-z} -\frac{x}{1 + x^2} \Big{)} x^2 d\nu(x)$ 
and $A_\lambda(z)= -\gamma' + \int \Big{(}\frac{1}{x-z} -\frac{x}{1 + x^2} \Big{)} x^2 d\nu'(x)$. 
Then we can show that 
$\lambda' = \lambda \gamma + \lambda \int \frac{(\lambda^2 - 1)x^3}{(1 + \lambda^2 x^2)(1 + x^2)} d\nu(x)$ and 
$\nu' = D_\lambda \nu$. Correspondingly, we shall use the L\'{e}vy-Khintchine formula 
$\widehat{\mu} (u) = \exp\Big{(}i\gamma u + \int_{\real}\big{(}e^{ixu} - 1 - \frac{ixu}{1 + x^2} \big{)}\nu(dx) \Big{)}$ and $\widehat{D_\lambda \mu} (u) = \exp\Big{(}i\gamma ' u + \int_{\real}\big{(}e^{ixu} - 1 - \frac{ixu}{1 + x^2} \big{)}\nu'(dx) \Big{)}$ for a probability measure $\mu \in ID(\ast)$. We can show the same expressions of $\gamma'$ and $\nu'$. \\
(6) This property follows from Theorem \ref{subordinator}. \\
(7) We first note that the following equality holds for $0 < \alpha < 2$: 
\begin{equation} 
\begin{split}
&\int_0 ^{\infty} \Big{(} e^{izx} - 1 - \frac{izx}{1 + x^2} \Big{)}\frac{1}{x^{1 + \alpha}}dx \\
&~~~~~~~~~~~~~~~~~~= \begin{cases} |z|^\alpha \Gamma(-\alpha)e^{-\frac{i}{2}\pi \alpha \sign z} - \frac{iz\pi}{2\cos(\frac{\pi \alpha}{2})} & \text{ for } \alpha \neq 1, z \in \real, \\ -\frac{\pi |z|}{2} - iz \log |z| + icz & \text{ for } \alpha = 1, z \in \real,  
        \end{cases}
\end{split} 
\end{equation}
where 
$c= \int_1 ^{\infty} \frac{\sin x}{x^2} dx + \int_0 ^{1} \frac{\sin x - x}{x^2}dx + \int_0 ^{\infty} \Big{(}x1_{(0,1)}(x) - \frac{x}{1 + x^2}\Big{)}dx$. 
This equality is obtained by Lemma 14.11 in \cite{Sat}. A necessary and sufficient condition for a strictly $\alpha $-stable distribution is given by: 
\begin{itemize}
\item[$(a)$]if $\alpha \neq 1$, 
\begin{equation}\label{str1}
\begin{split}
&\frac{1 + x^2}{x^2}\tau (dx) = \frac{c_1}{|x|^{1 + \alpha}}1_{(0, \infty)}(x)dx + \frac{c_2}{|x|^{1 + \alpha}}1_{(-\infty, 0)}(x)dx, c_1, c_2 \geq 0 \text{ and } \\
&\gamma = (c_1 - c_2) \frac{\pi}{2\cos(\frac{\alpha \pi}{2})}; 
\end{split}
\end{equation}
\item[$(b)$]if $\alpha = 1$, 
\begin{equation}\label{str2}
\frac{1 + x^2}{x^2}\tau (dx) = \frac{c}{x^2}dx , c \geq 0.  
\end{equation}
\end{itemize}
For the proof the reader is referred to Theorem 14.15 in \cite{Sat}. 
On the other hand, the vector field $A$ of monotone strictly $\alpha$-stable distribution is given by $A(z) = bz^{1-\alpha}$; 
the pair $(\gamma, \tau)$ appearing in $A$ is given by 
\begin{equation}
\begin{split}
\gamma = -\re A(i) = - \re (b e^{\frac{i(1- \alpha) \pi}{2}}) =  \im (b e^{-\frac{i \pi \alpha}{2}})
\end{split}
\end{equation}
and
\begin{equation}
\begin{split}
\frac{1 + x^2}{x^2} \tau(dx) &= \lim_{y \to + 0}\im b(x + iy)^{1-\alpha}\frac{1}{x^2} \\
                             &= \frac{\im b}{\pi} \frac{1}{|x|^{1 + \alpha}}1_{(0,\infty)}(x)dx +  \frac{\im (be^{i\pi (1-\alpha)})}{\pi} \frac{1}{|x|^{1 + \alpha}}1_{(-\infty, 0)}(x)dx.  
\end{split}
\end{equation}
When we define $c_1, c_2$ and $\gamma$ by $c_1 =\frac{\im b}{\pi}$, $c_2 = \frac{\im (be^{i\pi (1-\alpha)} )}{\pi}$ and $\gamma = \im (b e^{-\frac{i \pi \alpha}{2}}) $, 
we can show that the conditions (\ref{str1}) or (\ref{str2}) are satisfied by direct calculation. 
This fact implies that $\Lambda_M ^{-1}$ maps monotone strictly $\alpha$-stable distributions 
to strictly $\alpha$-stable distributions. We can also check that the correspondence is onto. \\
(8), (9) These properties are direct consequences of theorems \ref{symm} and \ref{moment19}. 
\end{proof}

\section{Monotone convolution and  Aleksandrov-Clark measures}\label{connect1}
In this section, we prove that monotone independence appears in the context of one-rank perturbations of a self-adjoint operator and a unitary operator. 
As a result, a family of probability measures called Aleksandrov-Clark measures are expressed by monotone convolutions. 
 
Let $H$ be a Hilbert space with a unit vector $|\Omega \rangle$. 
Let $X$ be a self-adjoint operator defined on a dense domain of a Hilbert space. We define the probability distribution $\nu$ of $X$ by 
\[
\langle \Omega|(z - X)^{-1}|\Omega \rangle = \int_{\real}\frac{1}{z - x} \nu(dx).  
\]
 If we define $\nu ^y$ to be the probability distribution of the self-adjoint operator $X^y:=X + yI$, $\nu^y$ is the translation of $\nu$. 
Therefore, the classical convolution $\mu \ast \nu$ is defined to be $\int \nu ^y d\mu(y)$. 

On the other hand, monotone convolutions can be characterized by the Aronszajn-Krein formula, which we now explain.  
 of the self-adjoint operator 

We denote by $\nu_y$ the probability distribution of $X_y := X + y|\Omega \rangle \langle \Omega|$; $\nu_y$ ($y \in \real$) are called the Aleksandrov-Clark measures of $\nu$. 
Aronszajn-Krein formula says that $\nu_y$ is characterized by $H_{\nu_y} = H_{\nu} - y$. For detailed properties of the Aronszajn-Krein formula, 
the reader is referred to \cite{Si-Wo} and \cite{Sim}. 
Therefore, we can understand the two kinds of convolutions $\ast, \rhd$ as the superpositions of perturbed probability distributions (see Eq. (\ref{eq21})).  
The difference between $\ast$ and $\rhd$ is the direction of perturbation: 
the usual convolution is perturbed by $yI$ and monotone convolution is perturbed by $y|\Omega \rangle \langle \Omega|$. 

We can prove the Aronszajn-Krein formula in terms of monotone independence. 
\begin{thm} Let $H$ be a Hilbert space with a normalized vector $|\Omega \rangle $. Let $\mathcal{B}(H)$ be the set of bounded operators on $H$. 
$|\Omega \rangle \langle \Omega|$ and $\mathcal{B}(H)$ are monotone independent w.r.t. the vector state $|\Omega \rangle $. 
\end{thm}
\begin{proof} Let $K$ be another Hilbert space with a unit vector $|\Omega' \rangle $. 
It is known \cite{Fra2} that $A \otimes |\Omega \rangle \langle \Omega|$ and $I \otimes B$ are monotone independent for any $A \in \mathcal{B}(K)$ and $B \in \mathcal{B}(H)$ w.r.t. the vector state $|\Omega' \rangle \otimes |\Omega \rangle $. This theorem follows from the case $A = I$. 
\end{proof}

This theorem implies that the Aronszajn-Krein formula can be seen as a special case of (\ref{mur11}). 

Moreover, a similar one-rank perturbation was introduced for a unitary operator (see \cite{Sim2} for details).  
Let $U$ be a unitary operator on $H$ and let $\lambda$ be the probability distribution of $U$ on the unit circle $\tor$; $\lambda$ is defined by 
\[
\langle \Omega| U^n|\Omega \rangle = \int_{\tor} \zeta^n d\lambda(\zeta). 
\]
The probability distribution of $Ue^{i\theta |\Omega \rangle \langle \Omega|}$ is also called an Aleksandrov-Clark measure. We denote this by $\lambda_{u, \theta}$. Let $M_\lambda(z):=\sum_{n = 1}^\infty \int_{\tor}\zeta^n z^nd\lambda(\zeta)$. If we define $\eta_{\lambda}(z):=\frac{M_\lambda(z)}{1 + M_\lambda(z)}$,  
$\lambda_{u, \theta}$ is characterized by 
\begin{equation}\label{alek2}
\eta_{\lambda_{u, \theta}} = e^{i\theta}\eta_{\lambda}.
\end{equation}
The reader is referred to the equality (1.3.92) of \cite{Sim2} (in this book, different functions are used). 

Now we give a proof of (\ref{alek2}) in terms of monotone independence. 
We notice that $e^{i\theta |\Omega \rangle \langle \Omega|} - 1 = (e^{i\theta} - 1)|\Omega \rangle \langle \Omega|$.
From the above theorem $e^{i\theta |\Omega \rangle \langle \Omega|} -1$ and $U$ are monotone independent for $\theta \in [0, 2\pi)$. This fact implies that $\lambda_{u,\theta}$ is equal to the multiplicative monotone convolution of $\delta_{e^{i\theta}}$ and $\lambda$. Therefore, (\ref{alek2}) can be seen to be a special case of the Bercovici's characterization of multiplicative monotone convolutions \cite{Ber1, Fra4}.  

Now we show some applications of this formula. 
In the following, we consider only self-adjoint operators.   

\textbf{1. Moments of convolution.}
$m_n(\mu \rhd \nu)$ can be calculated as 
$m_n(\mu \rhd \nu) = \int \langle \Omega | (X-y|\Omega \rangle \langle \Omega|)^n |\Omega \rangle d\mu(y)$. 
This calculation may be connected to the formula $H_{\mu \rhd \nu}(z) = H_{\mu} (H_\nu(z))$.

\textbf{2. Distributional properties of monotone convolution.} 
Many researchers have studied spectral properties of the perturbed self-adjoint operator $X_y$. 
Some of their results are applicable to monotone probability theory. 
For instance, Theorem 5 in \cite{Si-Wo} can be understood in terms of monotone convolution as follows.  
\begin{thm} $($B. Simon and T. Wolff$)$
Let $\nu$ be the Cauchy distribution $\nu(dx) = \frac{b}{\pi (x^2 + b^2)}dx$. Then $\nu \rhd \mu$ is mutually equivalent to Lebesgue measure for 
any probability measure $\mu$. 
\end{thm}

\textbf{3. Applications to Markov processes arising from monotone probability theory.}
Let $\mu_t$ be a monotone convolution semigroup. Let $\mu_{t,x}$ be a probability measure defined by $H_{\mu_{t, x}} = H_{\mu_t} - x$.
We can check that a family of probability measures $\{\mu_{t,x} \}_{t \geq 0, x \in \real}$ satisfies the Chapman-Kolmogorov equation, and hence, constitutes transition probability distributions of a usual Markov process, whose transition semigroup is given in \cite{F-M}. 
Then the spectral properties of $\mu_{t,x}$ are important when we 
try to study Markov processes arising from monotone L\'{e}vy processes. 

In addition, this realization of a convolution semigroup as a Markov process is not restricted to the case of  
continuous time processes; the discrete time version 
is also possible as stated below.
\begin{prop}\label{Markov1}
For a probability measure $\mu$, we define a family of probability measures $\{\mu_{n,x} \}_{n \in \nat, x \in \real}$ 
by 
\[
\mu_{n,x} = (\mu ^{\rhd n})_x = \mu_x \rhd \mu ^{\rhd n-1} 
\]
for $n \geq 1$ and $\mu_{0} = \delta_0$. Then the family $\{\mu_{n,x} \}_{n \in \nat, s \in \real}$ satisfies 
the Chapman-Kolmogolov equation $\mu_{n+m,x}(A) = \int_{\real} \mu_{n,y}(A)\mu_{m, x}(dy)$. 
Therefore, there exists a corresponding discrete time Markov process.  
\end{prop}  
We can, of course, apply the analyses of spectral properties of $X_y$ to study the discrete time Markov processes 
constructed in Proposition \ref{Markov1}. 
 
\textbf{4. Distributional properties of $V_a$-transformation.} 
$V_a$-transformation has been introduced in \cite{K-W} in the context of conditionally free convolutions. This transformation is identical to
the transformation $\mu \mapsto \mu_{-ar_\mu(2)}$ ($r_{\mu}(2)$ is a variance of $\mu$). 
Then we can apply the results of spectral analysis of $X_y$ to distributional properties of $V_a \mu$.

\section{Time-independent properties of boolean and free convolution semigroup }\label{time2}   
\subsection{Preliminaries}
We prepare important notions about free probability and boolean probability.
Notation is chosen in order that the correspondence becomes clear among Bercovici-Pata bijections in free, monotone and boolean probability theories.  
\begin{equation}
\begin{split}
K_\mu(z) &:= z - H_\mu (z) \\
          &= \gamma - \int\frac{1 + xz}{x-z} d\tau(x).
\end{split}\end{equation}
The boolean convolution $\mu \uplus \nu$ of probability distributions $\mu$ and $\nu$ is characterized by 
\begin{equation}
K_{\mu \uplus \nu} = K_\mu + K_\nu. 
\end{equation}
Now we explain infinitely divisible distributions in free probability theory.  
The reader is referred to \cite{BFGKT, Be-Vo}. 

For a probability measure $\mu$, there exists some $\eta > 0$ and $M > 0$ such that $H_\mu$ has an analytic right inverse $H_\mu ^{-1}$ defined on 
the region 
\[\Gamma_{\eta, M}:= \{z \in \comp; |\re z| < \eta |\im z|, ~|\im z| > M \}. \] 
The Voiculescu transform $\phi_\mu$ is defined by   
\begin{equation}
\phi_\mu(z):= H_\mu ^{-1}(z) - z
\end{equation}
on a region on which $H_\mu ^{-1}$ is defined. For probability measures $\mu$ and $\nu$, 
the free convolution of $\mu$ and $\nu$ is characterized by the relation 
\begin{equation}
\phi_{\mu \boxplus \nu} = \phi_\mu + \phi_\nu. 
\end{equation}

\begin{thm}
Let $\mu$ be a probability measure on $\real$. $\mu$ is $\boxplus$-infinitely divisible iff there exist a finite measure $\tau$ and a real number $\gamma$ 
such that 
\begin{equation}
\phi_\mu(z)= \gamma + \int_{\real} \frac{1 + xz}{z-x}d\tau(x)~~ \text{ for } z \in \comp \backslash \real.
\end{equation}
\end{thm}

\subsection{Free convolution semigroup and boolean convolution semigroup}
In Section \ref{time}, we have studied how a specific property of a monotone convolution semigroup changes as time passes. The important point is that a time independent property is sometimes characterized by the infinitesimal generator $A$; then it is probable that such a property is conserved by the map $\Lambda_M $. 
Therefore, properties of a convolution semigroup w.r.t. time parameter $t$ are 
important to study the connection between classical probability theory and another probability theory equipped with some notion of independence. 
Now we consider in the cases of boolean and free independence. 

First we show that the subordinator theorem is valid in the boolean case but is not valid in the free case.
\begin{thm}\label{subordinator2} 
Let $\{\mu_t \}_{t \geq 0}$ be a weakly continuous boolean convolution semigroup 
with $\mu_0 = \delta_0$. Then the following statements are equivalent: 
\begin{itemize}
\item[(1)] there exists $t_0 > 0$ such that $\supp\mu_{t_0} \subset [0, \infty)$;
\item[(2)] $\supp \mu_t \subset [0, \infty)$ for all $0 \leq t < \infty$;
\item[(3)] $\supp \tau \subset [0, \infty)$,  $\tau(\{0 \}) = 0$, $\int_0 ^{\infty}\frac{1}{x}d\tau (x) < \infty$ and 
$\gamma \geq \int_0 ^{\infty}\frac{1}{x}d\tau (x)$. 
\end{itemize} 
This type of theorem does not hold in free probability theory: 
Condition (1) is not equivalent to condition (2).
\end{thm}
\begin{rem}
The result is understood in terms of $t$-transform: $t$-transform preserves 
the positivity of a probability measure. 
\end{rem}

\begin{proof}
In the boolean case, the proof is easy by Proposition \ref{positive support2}. 
 
In the case of free probability theory, we show a counter example of convolution semigroup which does not satisfy the equivalence between $(1)$ and $(2)$. 
Since the problem is symmetric around the origin, we show a counter example concerning the condition $\supp \mu_t \subset (-\infty, 0]$.  
We define $\phi_\mu (z):= a -(z-c)^{\frac{1}{2}}$ with $a, c \in \real$. Then the corresponding convolution semigroup $\{ \mu_t \}_{t \geq 0}$ with $\mu_1 = \mu$, $\mu_0 = \delta _0$ is characterized by the inverse of the reciprocal Cauchy transform $H_t ^{-1}(z) = z + t\phi_\mu(z) = z +ta - t(z-c)^{\frac{1}{2}}$. 
Solving this equation, we obtain  
\begin{equation}
\begin{split}
 H_t(z) = z -at + \frac{t^2}{2} + t\sqrt{z-\Big{(}at - \frac{t^2}{4} + c \Big{)}}.
\end{split}
\end{equation}
The support of absolutely continuous part of $\mu_t$ is $(-\infty, at - \frac{t^2}{4} -c]$. 
If $a \geq 0$ and $a^2 > c$, then there exists $t_0 > 0$ such that $at - \frac{t^2}{4} -c > 0$ for $0 < t < t_0$ and 
$at - \frac{t^2}{4} -c < 0$  for $t > t_0$. Therefore, it holds that $\supp \mu_t \nsubseteq (-\infty, 0]$ for $0 < t < t_0$. 
On the other hand,  $H_t(+0) = -at + \frac{t^2}{2} + t \sqrt{-at + \frac{t^2}{4} - c} > 0$ for sufficiently large $t > t_0$. 
Then we have $\supp \mu_t \subset (-\infty, 0]$ for sufficiently large $t$. 
Therefore, we can conclude that the negativity of the support of a convolution semigroup is a time-dependent property.  
\end{proof}

The symmetry around the origin is a time-independent property also in the cases of boolean and free independence. 
\begin{prop}\label{symm2}
Let $\{\mu_t \}_{t \geq 0}$ be a weakly continuous boolean $($free$)$ convolution semigroup with $\mu_0 = \delta_0$.  
Then the following statements are all equivalent. 
\begin{itemize}
\item[$(1)$] There exists $t_0 > 0$ such that $\mu_{t_0}$ is symmetric.
\item[$(2)$] $\mu_{t}$ is symmetric for all $t > 0$.
\item[$(3)$] $\gamma = 0$ and $\tau$ is symmetric.
\end{itemize}
\end{prop}
\begin{proof}
In boolean probability theory, the symmetry of a probability distribution is equivalent to $K_\mu (-z) = -K_\mu(z)$ for all $z \in \comp \backslash \real$. 
In free probability theory, the symmetry of a probability distribution is equivalent to $\phi_\mu (-z) = -\phi_\mu(z)$ for all $z \in \Gamma_{\eta, M}$. 
This is clearly a time-independent property since the convolution means the addition of $K_\mu$ (resp. $\phi_\mu$) in boolean (resp. free) 
probability theory. In both cases, $(3)$ is equivalent to $(1)$ which has been pointed out in the boolean case in \cite{S-W}.  
\end{proof}

We can show that the property $\int_{\real} x^{2n} d\mu_t(x) < \infty $ is time-independent in boolean case. 
In free probability theory, this result is recently obtained in \cite{Ben}. 
We show that this is the case also in boolean probability theory. The proof is greatly easier than the case of monotone and free probability theories.  
\begin{prop}
Let $n \geq 1$ be a natural number. For a weakly continuous boolean convolution semigroup $\{\mu_t \}_{t \geq 0}$, 
the following statements are equivalent. 
\begin{itemize}
\item[$(1)$] $\int_{\real} x^{2n} d\mu_t(x) < \infty $ for some $t > 0$. 
\item[$(2)$] $\int_{\real} x^{2n} d\mu_t(x) < \infty $ for all $t > 0$.
\item[$(3)$] $\int_{\real} x^{2n} d\tau(x) < \infty $.
\end{itemize}
\end{prop}
\begin{proof} 
The proof follows from Proposition \ref{asymptotic2}. 
\end{proof}

Now we can compare the properties of Bercovici-Pata bijections in free, monotone and boolean probability theories. 
Boolean (strictly) stable distributions have been classified in \cite{S-W}, and they have the same characterization as monotone case. Considering the contents in this section, we obtain the boolean analogue of properties (1)-(9) in Theorem \ref{Lambda}. It might be interesting to consider the validity of property (6) in boolean and monotone cases 
d in terms of the embedding into tensor independence \cite{Fra1}. 
In free probability theory, most of the results of Theorem \ref{Lambda} are already known (see \cite{BFGKT, Be-Pa}) except for the failure of free analog of property (6).  

Finally, we note another similarity between free and monotone infinitely divisible distributions. The number of atoms in a $\boxplus$-infinitely divisible distribution is restricted in a similar way to the case of a $\rhd$-infinitely divisible distribution (see Theorem \ref{atom3} in this paper and Proposition 2.8 in \cite{Bel}).

\section{Examples}\label{Exa}
We consider the family of distributions $\mu_t ^{(\alpha,b,c)}$  introduced in (\ref{computa2}) under the following restrictions: $b = 1$ when $0 < \alpha < 1$ and $b = -1$ when $1 \leq  \alpha \leq  2$. In terms of the reciprocal Cauchy transform, we have 
\begin{equation}
H_t ^{(\alpha,1,c)}(z)= c + \{(z - c)^\alpha + t \}^{\frac{1}{\alpha}} \text{~ for~} 0 < \alpha < 1
\end{equation}
and  
\begin{equation}
H_t ^{(\alpha,-1,c)}(z)= c + \{(z - c)^\alpha - t \}^{\frac{1}{\alpha}} \text{~for~} 1 \leq \alpha \leq 2. 
\end{equation}
We write simply as $H_t(z)$ and $\mu_t$ when there are no confusions. 
The support of an absolutely continuous part and the positions of atoms for various parameters $\alpha, c, t ~(t > 0)$ 
are summarized as follows:   \\
(1) $\im c < 0$ \\
 \begin{equation}
\begin{split}
&\mu_t = \mu_{t, ac}, \\
&\supp \mu_{t, ac} = \real.
\end{split}
\end{equation}
(2) $\im c = 0$, $\re c \geq  0$, $\alpha = 2$  
\begin{equation}
\begin{split}
&\mu_t = \mu_{t, ac} + \mu_{t, sing}, \\
&\supp \mu_{t, ac} = [c - \sqrt{t}, c + \sqrt{t}], \\
&\mu_{t, sing} = \frac{|c|}{c^2 + t} \delta_{ c - \sqrt{c^2 + t} }.  
\end{split}
\end{equation}
(3) $\im c = 0$, $\re c < 0$, $\alpha = 2$ 
\begin{equation}
\begin{split}
&\mu_t = \mu_{t, ac} + \mu_{t, sing}, \\
&\supp \mu_{t, ac} = [c - \sqrt{t}, c + \sqrt{t}], \\
&\mu_{t, sing} = \frac{|c|}{c^2 + t} \delta_{ c + \sqrt{c^2 + t} }.
\end{split}
\end{equation}
(4) $\im c = 0$, $\re c \geq 0$, $1 < \alpha < 2$ 
\begin{equation}
\begin{split}
&\mu_t = \mu_{t, ac},  \\
&\supp \mu_{t, ac} = (-\infty, c + t^{\frac{1}{\alpha} }]. \\ 
\end{split}
\end{equation}
(5) $\im c = 0$, $\re c < 0$, $1 < \alpha < 2$ 
\begin{equation}
\begin{split}
&\mu_t = \mu_{t, ac} + \mu_{t, sing},  \\
&\supp \mu_{t, ac} = (-\infty, c + t^{\frac{1}{\alpha} }], \\ 
&\mu_{t, sing} = \Big{(}\frac{|c|^\alpha}{|c|^\alpha + t}\Big{)} ^ {\frac{\alpha - 1}{\alpha}} \delta_{ c + (|c|^\alpha + t)^{\frac{1}{\alpha}} }. 
\end{split}
\end{equation}
(6) $\im c = 0$, $\re c \geq 0$, $0 < \alpha < 1$ 
\begin{equation}
\begin{split}
&\mu_t = \mu_{t, ac},  \\
&\supp \mu_{t, ac} = (-\infty, c].  \\ 
\end{split}
\end{equation}
(7) $\im c = 0$, $\re c < 0$, $0 < \alpha < 1$  
\begin{equation}\label{exa7}
\begin{split}
&\mu_t = \mu_{t, ac} + \mu_{t, sing}, \\
&\supp \mu_{t, ac} = (-\infty, c], \\
&\mu_{t, sing} = 
     \begin{cases}
          \Big{(}\frac{|c|^\alpha - t}{|c|^\alpha}\Big{)}^{\frac{1 - \alpha}{\alpha}} \delta_{ c + (|c|^\alpha - t)^{\frac{1}{\alpha}} }, & 0 \leq t < |c|^\alpha, \\
          0, & t \geq |c|^\alpha.  
     \end{cases}
\end{split}
\end{equation}
(8) $\im c = 0$, $\alpha = 1$
\begin{equation}
\begin{split}
\mu_{t, sing} = \delta_{tc}. 
\end{split}
\end{equation}

Several features can be seen from the above examples. In (6), the support of the absolutely continuous part does not vary as a function of $t$; 
however, it varies as a function of $t$ in (5), for instance.   
One can see in (7) that there exists a probability measure $\mu$ which contains a delta measure 
although $\mu \rhd \mu$ does not contain a delta measure. Therefore, we can conclude that 
monotone convolution does not conserve the absolute continuity of probability distributions. 

We calculate  $\mu_t$ explicitly in the case of $\alpha = 2$ and $\im c = 0$, and in the case of 
$\alpha = \frac{1}{2}$ and $\im c = 0$. 
When $\alpha=2$ and $\im c = 0$, the result is shown in \cite{Mur3}. Coefficients of delta measures 
are, however, not shown in \cite{Mur3}. It is important from the viewpoint of Theorem \ref{atom3} to check that two delta measures do not 
appear at the same time. 

\noindent
\textbf{\underline{1. $\alpha = 2$  and $\im c = 0$}} 
The Cauchy transform is given by
\[G_t (z) = \Big{[}c + |(x-c)^2 - t - y^2+ 2iy(x-c)|^{\frac{1}{2}} \exp \Big{\{}\frac{i}{2}\arg \big{(}(x-c)^2 - t - y^2+ 2iy(x-c)\big{)} \Big{\}}  \Big{]}^{-1}. \]  \\
$\mathbf{case ~1}$: $|x-c| > \sqrt{t}$ \\
Since $(x-c)^2 - t > 0$, we have 
\begin{equation}
\lim_{y \rightarrow + 0}\exp \Big{\{}\frac{i}{2}\arg \big{(}(x-c)^2 - t - y^2+ 2iy(x-c)\big{)} \Big{ \} } = 
   \begin{cases} 
         1, & x-c > + \sqrt{t}, \\
        -1, & x-c < - \sqrt{t}.
    \end{cases}
\end{equation}
Therefore, we have the limit 
\begin{equation}
\lim_{y \rightarrow + 0} G_t (x + iy) = 
   \begin{cases} 
         \frac{1}{c + \sqrt{(x-c)^2 - t} }, & x-c > + \sqrt{t}, \\
         \frac{1}{c - \sqrt{(x-c)^2 - t} }, & x-c < - \sqrt{t}.
    \end{cases}
\end{equation}

\noindent
$\mathbf{case ~2}$: $|x-c| \leq  \sqrt{t}$ \\ 
 In this case we have
\begin{equation}
\lim_{y \rightarrow + 0}\exp \Big{\{}\frac{i}{2}\arg \big{(}(x-c)^2 - t - y^2+ 2iy(x-c)\big{)}\Big{ \}} = i.
\end{equation}
Hence we obtain 

\begin{equation}
\begin{split}
      \lim_{y \rightarrow + 0} G_t (x + iy) &= \frac{1}{c + i\sqrt{t - (x-c)^2} }  \\
                                            &= \frac{c - i\sqrt{t - (x-c)^2}}{c^2 + t - (x-c)^2}. 
\end{split}  
\end{equation}
The absolutely continuous part of $\mu_t$ is 
\begin{equation} 
   d\mu_{t,ac}(x) =  \frac{1}{\pi }\frac{\sqrt{t - (x-c)^2}}{c^2 + t - (x-c)^2} 1_{(c-\sqrt{t}, c+ \sqrt{t})}(x)dx. 
\end{equation}
There is a delta measure at $a$ only when $\lim_{y \searrow  0}iyG_t(a+iy) > 0$.  
As a possible position of a delta measure, there is only one point $a$ with $|x-c| > \sqrt{t}$ which satisfies 

\[ \begin{cases}    c - \sqrt{(a-c)^2 - t} = 0 \text{~and~} a-c < -\sqrt{t}, & \text{~if~} c \geq 0,  \\
                    c + \sqrt{(a-c)^2 - t} = 0 \text{~and~} a-c > \sqrt{t}, &\text{~if~} c < 0,  
   \end{cases}
\]
that is, $a = c - \sqrt{c^2 + t}$ for $c \geq 0$ and $a = c + \sqrt{c^2 + t}$ for $c < 0$. 
Therefore, the singular part of $\mu_t$ is 
\begin{equation} \mu_{t,sing} = \begin{cases} 
                   A\delta_{ c - \sqrt{c^2 + t} }, & c \geq 0,  \\
                   B\delta_{c + \sqrt{c^2 + t} }, & c < 0.  \\  
                  \end{cases}  
\end{equation}
It is not difficult to calculate $\mu_{t, ac}(\real) = 1- \frac{|c|}{\sqrt{c^2 + t}}$. 
Then one can determine the singular part completely: 
\begin{equation} \mu_{t,sing} = \begin{cases} 
                   \frac{|c|}{\sqrt{c^2 + t}} \delta_{ c - \sqrt{c^2 + t} }, & c \geq 0,  \\
                   \frac{|c|}{\sqrt{c^2 + t}} \delta_{c + \sqrt{c^2 + t} }, & c < 0.  \\  
                  \end{cases}  
\end{equation}

\noindent
\textbf{\underline{2. $\alpha=\frac{1}{2}$ and $\im c = 0$}} 
The reciprocal Cauchy transform is given by
\begin{equation}
\begin{split}
          H_t (z) &= c + \{(z - c)^{\frac{1}{2} } + t \}^2   \\
                  &= t^2 + z + 2t\sqrt{z-c}. 
\end{split}
\end{equation}
Here the branch of $(z - c)^{\frac{1}{2} }$ is taken such that $\sqrt{-1} = i$.  
If $x-c > 0$, then
\begin{equation} G_t (x + i0) = \frac{1}{ t^2 + x + 2t\sqrt{x-c} }. 
\end{equation}
 If $x-c \leq 0$, then 
\begin{equation}
\begin{split}
   G_t (x + i0) &= \frac{1}{ t^2 + x + 2ti\sqrt{c-x} }   \\
                &= \frac{t^2 + x - 2ti\sqrt{c-x} }{(t^2 +x)^2 + 4t^2 (c-x)}.
\end{split}
\end{equation}
To know the position and the weight of a delta measure it is important to calculate the following quantity: 
\begin{equation}
\lim_{y \rightarrow +0} y G_t (a + iy) = 
           \begin{cases} 
            \frac{\sqrt{|c|} -t}{i\sqrt{|c|}}, &\text{if~} a = t^2 - 2t\sqrt{|c|},~ c < 0,~  t\leq \sqrt{|c|},  \\
                               0, & \text{~~otherwise}.      
           \end{cases}
\end{equation} 
If $c < 0$, there is a delta measure in this distribution at $x = t^2 - 2t\sqrt{|c|}$ and it disappears 
at the time $t = \sqrt{|c|}$. The maximum $b(\mu_t)$ of the support of $\mu_t$ 
is 

\[ b(\mu_t) = 
          \begin{cases} 
             t^2 - 2t\sqrt{|c|}, & 0 \leq t \leq \sqrt{|c|}, \\ 
             c, & t \geq \sqrt{|c|}. 
          \end{cases}
\]  
$b(\mu_t)$ decreases as a function of $t$.
The absolutely continuous part is calculated as 
\[ 
d\mu_{t,ac}(x)= \frac{1}{\pi}\frac{2t\sqrt{c-x} }{(t^2 +x)^2 + 4t^2 (c-x)} 1_{(-\infty, c)}(x)dx. 
\]
Convergence of $\mu_t$ as $t \rightarrow 0$ is
\begin{equation}
 \lim_{t\rightarrow 0} \mu_t =  
    \begin{cases} 
            \delta_0, & c \geq   0,          \\
            0, & c < 0.
    \end{cases}
\end{equation}

Finally, the distribution is obtained as follows: 
  
\begin{equation}   \mu_t = 
 \begin{cases} 
            \mu_{t,ac}, & c \geq  0, ~ 0 < t < \infty,        \\
            \mu_{t,ac} + \frac{\sqrt{|c|} -t}{\sqrt{|c|}} \delta_{t^2 - 2t\sqrt{|c|}}, & c < 0,~ 0 < t \leq \sqrt{|c|},    \\
            \mu_{t,ac}, & c < 0,~ t \geq \sqrt{|c|}.
 \end{cases}
\end{equation}

\section*{Acknowledgments} The author expresses his great appreciation to 
Prof. Izumi Ojima for his warm encouragements and valuable suggestions during his master course. 
He thanks Prof. Nobuaki Obata for giving him an opportunity to present this research in Tohoku University. 
He is also grateful to Prof. Shogo Tanimura, Mr. Hayato Saigo, Mr. Ryo Harada, Mr. Hiroshi Ando, Mr. Kazuya Okamura 
for their interests in his research and many helpful discussions. 
This work was partially supported by Grant-in-Aid for
JSPS research fellows.

\end{document}